\documentclass[a4paper,11pt,oneside]{amsart}
\pdfoutput=1

\usepackage[utf8]{inputenc}
\usepackage{bm}
\usepackage{mathtools,amssymb}
\usepackage{esint}
\usepackage{hyperref}
\hypersetup{
    colorlinks=true,
    linkcolor=black,
    filecolor=black,      
    urlcolor=cyan,
    citecolor=black
    }
\usepackage{tikz}
\usepackage{dsfont}
\usepackage{relsize}
\usepackage{url}
\usepackage{xcolor}
\usepackage{graphicx}
\usepackage{mathrsfs}
\usepackage[shortlabels]{enumitem}
\usepackage{lineno}
\usepackage{amsmath}
\usepackage{enumitem}
\usepackage{amsthm} 
\usepackage{fullpage} 
\usepackage{verbatim}
\usepackage{dsfont}

\DeclarePairedDelimiter\floor{\lfloor}{\rfloor}

\allowdisplaybreaks

\mathtoolsset{showonlyrefs}

\graphicspath{{images/}}

\newtheorem{theorem}{Theorem}[section]
\newtheorem{lemma}[theorem]{Lemma}
\newtheorem{definition}[theorem]{Definition}

\newtheorem{proposition}[theorem]{Proposition}

\newtheorem{remark}[theorem]{Remark}

\title[Anisotropic fractional conductivity]{Uniqueness for the anisotropic fractional conductivity equation}
\keywords{Fractional Laplacian, fractional gradient, Calderón problem, conductivity equation}
\subjclass[2020]{Primary 35R30; secondary 26A33, 42B37, 46F12}

\author{Giovanni Covi}
\address{Institut fur Angewandte Mathematik, Ruprecht-Karls-Universit\"at Heidelberg, Im Neuenheimer Feld 205, 69120 Heidelberg, Germany}
\email{giovanni.covi@uni-heidelberg.de}
\date{\today}

\newcommand{\C}{{\mathbb C}}
\newcommand{\R}{{\mathbb R}}

\newcommand{\N}{{\mathbb N}}

\newcommand{\eps}{\varepsilon}

\newcommand{\dimens}{n}

\newcommand{\aabs}[1]{\left\lVert #1 \right\rVert}

\newcommand{\Sym}{\mbox{Sym}}

\usepackage{amsmath}
\usepackage{bm} 

\newcommand{\tridots}{%
  \mathinner{\vcenter{
    \normalbaselines
    \baselineskip=\dimexpr\fontcharht\textfont0`:-\fontcharht\textfont0`.\relax
    \hbox{.}\hbox{.}\hbox{.}
  }}
}

\begin{document}

\maketitle
\begin{abstract}
In this paper we study an inverse problem for fractional anisotropic conductivity. Our nonlocal operator is based on the well-developed theory of nonlocal vector calculus, and differs substantially from other generalizations of the classical anisotropic conductivity operator obtained spectrally. We show that the anisotropic conductivity matrix can be recovered uniquely from fractional Dirichlet-to-Neumann data up to a natural gauge. Our analysis makes use of techniques recently developed for the study of the isotropic fractional elasticity equation, and generalizes them to the case of non-separable, anisotropic conductivities. The motivation for our study stems from its relation to the classical anisotropic Calder\'on problem, which at the time of writing is one of the main open problems in the field.
\end{abstract}

\section{Introduction}

Let $s\in(0,1)$. We study the anisotropic fractional conductivity operator
$$ \mathbf C^s_A u := (\nabla\cdot)^s (A(x,y)\cdot \nabla^s u), $$
where $A\in L^\infty(\R^{2n}, \R^{n\times n})$ is an $n\times n$ matrix function of the variables $x,y\in\R^n$. The operators $(\nabla\cdot)^s, \nabla^s$, which we will define in detail in the upcoming Section \ref{sec-prel}, are the fractional divergence and gradient. They were firstly introduced in the nonlocal vector calculus presented in \cite{DGLZ12}, and their properties were further studied in \cite{C20, C20a, CRZ22}, among others. They represent fractional counterparts to the classical divergence and gradient operators. As a matter of fact, they enjoy the expected properties $(\nabla^s)^* = (\nabla\cdot)^s$ and $(\nabla\cdot)^s\nabla^s = (-\Delta)^s$. The operator $\mathbf C^s_A$ will be rigorously defined in Section \ref{sec-oper}.
\\

 Our inverse problem consists in recovering the values of $A$ in $\Omega$ from measurements performed in the exterior $\Omega_e:=\R^n\setminus\overline\Omega$. These measurement will be given in the form of a Dirichlet-to-Neumann (DN) map $\Lambda_A$, associating each exterior datum $f$ of the direct problem
\begin{equation*}
    \begin{array}{rll}
       \mathbf C^s_A u & =0& \quad \mbox{ in } \Omega \\
       u & = f & \quad \mbox{ in } \Omega_e
    \end{array}
\end{equation*}
to the corresponding nonlocal Neumann datum $\mathbf C^s_Au_f|_{\Omega_e}$. Here $u_f\in H^s(\R^n)$ is the unique weak solution of the above problem associated to $f$. We will show that there exists in fact only one such $u_f$ in Section \ref{sec-wellp}, and thus $\Lambda_A$ is well-defined. In order to represent the inaccessibility to the measurements of certain regions of space, which is a very common physical situation, we assume that $\Lambda_Af$ is known only on $W_2$ and only for $f\in C^\infty_c(W_1)$, where $W_1, W_2 \subset\Omega_e$ are non-empty, open and disjoint subsets of the exterior. 
\\

The inverse problem we are interested in can be formulated as follows:
\\

\textbf{Q1:} \emph{Does $\Lambda_{A_1}f|_{W_2}=\Lambda_{A_2}f|_{W_2}$ for all $f\in C^\infty_c(W_1)$ imply $A_1=A_2$?}
\\

However, by Remark \ref{natural-gauge} it becomes soon evident that the answer to question \textbf{Q1} can not be positive, since the anisotropic fractional conductivity operator $\mathbf C^s_A$ possesses a natural gauge related to the properties of symmetry of the matrix $A$. Therefore, we shall rather ask: 
\\

\textbf{Q2:} \emph{Does $\Lambda_{A_1}f|_{W_2}=\Lambda_{A_2}f|_{W_2}$ for all $f\in C^\infty_c(W_1)$ imply $A_{1,s}=A_{2,s}$?}

Here $A_s$ is a symmetrized version of the matrix $A$, in a sense that will be made clear in the same Remark \ref{natural-gauge}. 
\\

The study of the fractional Calder\'on problem dates back to the seminal article \cite{GSU20}, which considered the recovery of a bounded potential for the fractional Schr\"odinger equation in a bounded domain from measurements performed in the exterior in the form of a nonlocal DN map. This result was then generalized in many different directions by the works \cite{BGU21, GLX17, GRSU20, C21, CGFR22, CMRU22, Li20a, Li20b, Li21, KRS21, RS19}, to cite a few. These contributions address the problem of uniqueness in the presence of both local and nonlocal perturbations, the problem of stability, and that of reconstruction. The inverse problem we study represents a generalization of the inverse problem for the \emph{isotropic} fractional conductivity equation, which was first introduced in \cite{C20a} and later studied quite extensively in the works \cite{C20, CRZ22, CRTZ22, LRZ22, RZ22}, among others. 

Our study is also related to the inverse problem for the anisotropic \emph{classical} conductivity equation. This is one of the main open problems in the field; it is only well understood at the time of writing in dimension $n=2$ (see \cite{LU01}), and only partial results are available in dimension $n\geq 3$ (see \cite{LTU03, LU89, DSFKSU07, DKLS16}). 

The very recent article \cite{FGKU21} also bears a relation to ours, in that the authors of this work also consider the problem of uniqueness for a fractional anisotropic conductivity operator. This is however obtained spectrally, and thus generalizes the classical anisotropic conductivity operator in a different way than done here. Therefore, the used techniques and obtained results differ substantially.\\

In order to answer question \textbf{Q2}, we attempt to adopt the technique from \cite{GSU20}. We first prove the Alessandrini identity, that is, an integral identity which relates the difference of the DN maps $\Lambda_{A_1}, \Lambda_{A_2}$ to the differences of the matrices $A_1, A_2$ using special solutions to the direct problem. The left-hand side of the Alessandrini identity is then equated to $0$ by the assumption on the DN maps. By testing the resulting identity with particular solutions, we are able to deduce the desired result. In order to produce such solutions we prove a Runge approximation property, using a technique involving the Hahn-Banach theorem and the unique continuation property (UCP) for the fractional Laplacian. 
\\

However, as observed in \cite[Remark 7.3]{CdHS22}, this technique will not lead to the wanted result by itself. This is due to the fact that the operator appearing in the right-hand side of the Alessandrini identity is nonlocal, which causes problems with the supports of both solutions and data. This issue may be avoided by means of the so called \emph{fractional Liouville reduction}, which transforms a conductivity-type problem in a Schr\"odinger-type one. This technique, whose name and idea come from the method used for the classical conductivity equation, was introduced in \cite{C20} for the \emph{isotropic} fractional conductivity equation and later generalized in \cite{CdHS22} for studying the inverse problem for fractional linear elasticity. Once the reduction is performed (see Section \ref{sec-wellp}), the direct problem takes the form
\begin{equation}
    \begin{split}
        \mathbf (-\Delta)^{s-1}D w + w\cdot Q = 0 & \quad\mbox{ in } \Omega \\
w = g & \quad\mbox{ in } \Omega_e
    \end{split},
\end{equation}
where the new solution $w$ and the new exterior datum $g$ are respectively computed from $u$ and $f$. Here the transformed potential $Q$ contains all the information about the matrix $A$, and $D$ is a fixed second order differential operator. As it will be shown in Section \ref{sec-mainproof}, the technique based on the Runge approximation property explained above works for the transformed problem. This is due to the fact that the relative Alessandrini identity has a local operator on the right-hand side, namely the matrix multiplication of $Q_1-Q_2$ against two special solutions of the transformed problem.
\\

There is however one more obstacle which needs to be addressed. In both \cite{C20} and \cite{CdHS22} the coefficient part of the considered unreduced operators, respectively the matrix $\gamma^{1/2}(x)\gamma^{1/2}(y)$Id and the tensor $C^{1/2}(x):C^{1/2}(y)$, were by definition separable functions of the variables $x$ and $y$. This detail is of great importance at the time of proving the fractional Liouville reduction. Since our matrix $A(x,y)$ is not in general a separable function of the variables $x$ and $y$, we are prevented from applying the fractional Liouville reduction directly. However, in Section \ref{sec-oper} we shall prove the existence of a decomposition of the kind
\begin{equation}
    A_s(x,y) = \sum_k \Phi_k(x) \odot \Phi_k(y),
\end{equation}
where $\{\Phi_k\}_{k\in\N}$ is a suited sequence of matrices depending on one variable, and $\odot$ indicates the Hadamard product (see Section \ref{sec-prel} for the definition of this and other matrix operations). Once this formula is achieved, it will be possible to show that a fractional Liouville reduction holds for each element of the above sum. As a consequence, the functions $w, Q$ and $g$ in the reduced problem will rather be sequences. 
\\

In order to obtain a suitable Alessandrini identity for the transformed problem, we need to additionally assume that the anisotropic coefficients are in gauge in the following sense: 

\begin{definition}\label{def:gauge}
 Let $\Omega\subset\R^n$ be open and bounded. We say that $A_1, A_2\in L^\infty(\R^{2n}, \R^{n\times n})$ are in gauge if there exists $\rho\in C^\infty_c(\Omega)$ such that $A_{2,s}(x,y)=(1+\rho(x))(1+\rho(y))A_{1,s}(x,y)$ for all $x,y\in \R^n$. In this case we write $A_1\sim A_2$.
\end{definition}

Using the techniques described above, we have obtained the following uniqueness result:

\begin{theorem}\label{main-theorem}
Let $\Omega, W_1, W_2\subset\mathbb R^n$ be bounded open sets such that $W_1, W_2\subset\Omega_e$, and assume $s\in(0,1)$. Let $A_1,A_2\in L^\infty(\R^{2n}, \R^{n\times n})$ be two anisotropic coefficients satisfying assumptions \ref{ass-exterior}--\ref{ass-Phi} and \ref{ass-isotropic}. Then $A_1\sim A_2$ and 
$$ \Lambda_{A_1}f|_{W_2} = \Lambda_{A_2}f|_{W_2} \qquad \mbox{ for all } f\in C^\infty_c(W_1)$$
hold if and only if $A_{1,s}=A_{2,s}$. 
\end{theorem}

Assumptions \ref{ass-exterior}--\ref{ass-Phi} describe the H\"older regularity of $A$, its behaviour in $\R^{2n}\setminus\Omega^2$, and a uniform positivity condition which is required for the well-posedness of the direct problem (see Section \ref{subsec-ass}). Assumption \ref{ass-isotropic} is only needed for the proof of the Runge approximation property, and it is an isotropicity requirement for the exterior values of the coefficients (see Lemma \ref{better-runge-2}). 

\subsection{Organization of the rest of the article}
The article is organized as follows. The first section is the Introduction, which presents the main problem, its connection to the literature, the statements of the main results and the techniques involved in their proofs. The second section is dedicated to the necessary Preliminaries and definitions, in particular about tensors, fractional operators, H\"older and Sobolev spaces. Section 3 discusses our assumptions for the anisotropy matrix, introduces the anisotropic fractional conductivity operator, and proves our fundamental reduction result. The fourth section treats the problem of well-posedness both for our original problem and for the transformed one, presenting the fractional Liouville reduction related to our discussion. The DN map and Alessandrini identity are both presented in Section 5, while Section 6 contains the Runge approximation property and the proof of the main theorem. Finally, the last section 7 gives the relation between the classical and fractional problems of anisotropic conductivity in the limit case $s\rightarrow 1$. 

\section{Preliminaries}\label{sec-prel}

\subsection{Operations on tensors}
In this section we introduce a number of operations between tensors and sequences of tensors. Define $\N$ as the set of positive integers, $\N_0:=\N\cup\{0\}$ as the set of non-negative integers, and $\N_\infty:= \N \cup\{\infty\}$ as the set of \emph{extended} positive integers. In all the definitions below, we assume that $m,n \in \mathbb N$, that $a\in \mathbb N_{\infty}^m, b\in \mathbb N_{\infty}^n$ are two (extended) multi-indices, and that $A,B$ are two (extended) tensors respectively given by $A_{\alpha_{1}, ... , \alpha_{m}}$ and $B_{\beta_{1}, ... , \beta_{n}}$, with $\alpha_i \in \{1, ..., a_i\}$ for $i=1, ..., m$ and $\beta_j \in \{1, ..., b_j\}$ for $j=1, ..., n$.

\begin{definition}[Tensor product]\label{def-tensor}
The tensor $A\otimes B$ of elements
$$(A\otimes B)_{\alpha_1, ..., \alpha_m, \beta_1, ..., \beta_n} := A_{\alpha_{1}, ... , \alpha_{m}}B_{\beta_{1}, ... , \beta_{n}}$$ is the \emph{tensor product} of $A$ and $B$. If $B$ is a scalar, in which case $n=b_1=1$, we simply write $AB$ for $A\otimes B$.
\end{definition}

\begin{definition}[Hadamard product]\label{def-hadamard}
Assume $m=n$ and $a=b$. The tensor $A\odot B$ of elements
$$(A\odot B)_{\alpha_1, ..., \alpha_m} := A_{\alpha_{1}, ... , \alpha_{m}}B_{\alpha_{1}, ... , \alpha_{m}}$$ is the \emph{Hadamard product} of $A$ and $B$.
\end{definition}

\begin{definition}[Tensor contraction of order $k$]
 Let $k \in \mathbb N$ verify $k\leq \min \{m,n\}$, and let $a_{m+\ell-k} = b_{\ell}$ for all $\ell \in \{1, ..., k\}$. The tensor $A\cdot_k B$ of elements 
$$ (A\cdot_k B)_{\alpha_1, ..., \alpha_{m-k}, \beta_{k+1}, ..., \beta_n} := \sum_{\beta_1=1}^{b_1}...\sum_{\beta_k = 1}^{b_k}A_{\alpha_1, ..., \alpha_{m-k}, \beta_1, ..., \beta_k} B_{\beta_1, ..., \beta_n}$$
is the \emph{$k$-th contraction} of $A$ and $B$. For simplicity we use $\cdot$ for $\cdot_1$, $:$ for $\cdot_2$, and $\tridots$ for $\cdot_3$.
\end{definition}

A sequence $\{A_k\}_{k\in\N}$ of (possibly extended) tensors can itself be seen as an extended tensor $A_{k, \alpha_1, ... \alpha_n}$ with one more index $k$ with range in $\N$. With this identification in mind, all the above definitions can be extended to sequences of tensors. They can also be naturally extended to functions and function spaces (see \cite{CdHS22} for more properties relative to these tensor operations).

\begin{remark}
It is customary to interpret the symbols $\nabla u, \nabla\cdot u$ for the gradient and divergence of a function respectively as a tensor and a scalar product (i.e. a contraction), where $\nabla$ formally represents a vector of partial derivatives. Building on this, we use the symbol $\nabla^2u := (\nabla\otimes\nabla) u$ for the Hessian matrix of $u$. Similarly, with clear meaning of the symbols, $\nabla^2\odot$ is the differential operator mapping an $n\times n$ matrix $M$ to $\nabla^2\odot M$, with $(\nabla^2\odot M)_{ij} := \partial_i\partial_j M_{ij}$.

Observe that here and everywhere else we are \underline{\emph{not}} assuming the Einstein summation convention on repeated indices.
\end{remark}

\subsection{Fractional Sobolev spaces}
Let $\hat u(\xi) = \mathcal{F}u(\xi) := \int_{\mathbb R^n} e^{-ix\cdot\xi}u(x)\,dx$ indicate the Fourier transform. If $r$ is any real number, we indicate the usual $L^2$-based Bessel potential spaces by $H^r(\R^n)$, with norm
$$\|u\|_{H^{r}} := \|(1+|\xi|^2)^{r/2}\hat u(\xi)\|_{L^2(\R^n)}.$$

\noindent Let $\Omega, F\subset\R^n$ respectively be an open and a closed set. Following \cite{McL00} and \cite{GSU20}, we define the fractional Sobolev spaces
\begin{align*}
    H^{r}(\Omega) & := \{ u|_\Omega , u\in H^{r}(\mathbb R^n) \} ,\\
    \widetilde H^{r}(\Omega) & := \,\mbox{closure of } C^\infty_c(\Omega) \mbox{ in } H^{r}(\mathbb R^n), \\
    H^{r}_F(\mathbb R^n) & := \{u\in H^{r}(\mathbb R^n) : \mbox{ supp}(u)\subset F\}. 
\end{align*}
\noindent As usual, the quotient norm $\|u\|_{H^{r}(\Omega)}:= \inf\{ \|w\|_{H^{r}} : w\in H^{r}(\mathbb R^n), w|_{\Omega} = u \}$ is used for $H^{r}(\Omega)$. The spaces defined above are related to each other in the following ways:
$$ \widetilde H^r(\Omega) \subseteq H^r(\Omega)\,,\quad \widetilde H^r(\Omega) \subseteq H^r_{\overline\Omega}\,,\quad H^r_F \subseteq H^r\,,$$
$$ (\widetilde H^r(\Omega))^* = H^{-r}(\Omega)\,,\quad (H^r(\Omega))^* = \widetilde H^{-r}(\Omega). $$
Moreover, if $\partial\Omega$ is Lipschitz, then $\widetilde H^r(\Omega) = H^r_{\overline\Omega}$ for all $r\in \mathbb R$ (see \cite{CMRU22, McL00}).

\subsection{Fractional operators}

Let $\mathcal S$ be the set of Schwartz functions, and assume $s\in\mathbb R^+\setminus\mathbb Z$. The \emph{fractional Laplacian} is the continuous map from $\mathcal S$ to $L^\infty$ defined by
$$(-\Delta)^s u := \mathcal{F}^{-1}(|\xi|^{2s}\hat u(\xi))$$
for $u\in\mathcal S$. One can uniquely extend this definition to have $(-\Delta)^s : H^r(\mathbb R^n)\rightarrow H^{r-2s}(\mathbb R^n)$ continuously for all $r\in\mathbb R$. It is also possible to have extensions to Sobolev spaces of negative exponent and to $L^p$-based Sobolev spaces (see e.g. \cite{CMR21, GSU20}). Alternatively, one can define the fractional Laplacian as the singular integral $$ (-\Delta)^su(x) := C_{n,s}\,PV\int_{\mathbb R^n}\frac{u(x)-u(y)}{|x-y|^{n+2s}}dy ,$$
where $C_{n,s}:= \frac{4^s \Gamma(n/2+s)}{\pi^{n/2}|\Gamma(-s)|}$. Many other equivalent definitions are also possible (see \cite{Kw17}). The fractional Laplacian verifies the UCP (unique continuation property): If $u\in H^r(\mathbb R^n)$ is such that $u=(-\Delta)^su=0$ in a non-empty open set $V$, then $u\equiv 0$ in $\mathbb R^n$. The proof of this fact can be found in \cite{GSU20} for $s\in(0,1)$ and in \cite{CMR21} for higher exponents. The latter paper also proves the following fractional Poincar\'e inequality: If $s\geq t\geq 0$ and $u\in H_K^s(\R^\dimens)$ for a compact set $K\subset\R^n$, then there exists a constant $\widetilde{c}=\widetilde{c}(n, K, s)> 0$ such that
\begin{equation*}
\aabs{(-\Delta)^{t/2}u}_{L^2(\R^\dimens)}\leq \widetilde{c}\aabs{(-\Delta)^{s/2}u}_{L^2(\R^\dimens)}.
\end{equation*}

We also need the fractional gradient and divergence operators $\nabla^s$ and $(\nabla\cdot)^s$. These were first introduced in the nonlocal vector calculus developed in \cite{DGLZ12, DGLZ13}, and were later studied in their present form in \cite{C20, C20a, CMR21, CRZ22, CdHS22} and more. Let $u\in C^\infty_c(\mathbb R^n)$ and $s\in (0,1)$. The fractional gradient of $u$ is the two-points function given by 
$$ \nabla^s u (x,y) := (u(y)-u(x))\otimes \zeta(x,y), \quad\mbox{ where } \zeta(x,y):=\frac{C_{n,s}^{1/2}}{\sqrt{2}} \frac{x-y}{|x-y|^{n/2+s+1}}.$$ This definition can be uniquely extended to act as $H^s(\mathbb R^n)\rightarrow L^2(\mathbb R^{2n})$. The adjoint of $\nabla^s$, which is indicated by $(\nabla\cdot)^s$ and acts as $L^2(\mathbb R^{2n})\rightarrow H^{-s}(\mathbb R^n)$, is by definition the fractional divergence. One can observe that the property $$(\nabla\cdot)^s\nabla^s u = (-\Delta)^s u$$ holds in $H^s(\mathbb R^n)$. Because of their mapping properties and of the particular relation they have among themselves and with the Laplacian, the fractional gradient and divergence can be though of as nonlocal counterparts of the classical gradient and divergence operators.

\subsection{H\"older spaces} \label{holder-section} In this Section we introduce the H\"older spaces and their main properties, following \cite{Ta96} (see also \cite{CdHS22}). If $r\in[0,1)$ and $k\in\N_0$, we define
\begin{gather}
    C^r(\mathbb R^n) :=\{ u: \R^n \rightarrow \C \;:\; u \mbox{ bounded and } |u(x)-u(y)|\leq C|x-y|^{r} \mbox{ for all } x,y\in\mathbb R^n\}, \\ C^k(\mathbb R^n):= \{ u: \R^n \rightarrow \C \;:\; D^\beta u \mbox{ bounded continuous for all } |\beta|\leq k\}, \\ C^{r+k}(\mathbb R^n):=\{ u\in C^k(\mathbb R^n) \;:\; D^\beta u \in C^r(\mathbb R^n) \mbox{ for all } |\beta|=k \}.
\end{gather}
The H\"older norm of $u\in C^{r+k}(\R^n)$ is defined as
$$ \|u\|_{C^{r+k}(\R^n)} := \max_{|\beta|\leq k}\|D^\beta u\|_{L^\infty(\R^n)} + \max_{|\beta|=k}\sup_{x\neq y}\frac{|D^\beta u(x)-D^\beta u(y)|}{|x-y|^r}. $$
Moreover, the following properties are known to hold:
\begin{enumerate}[label=(H\arabic*)]
    \item\label{H1} Continuous embedding property:
    
    \noindent If $r,r'\in\R^+$ are such that $r\leq r'$, then $C^{r'}(\mathbb R^n)\hookrightarrow C^r(\mathbb R^n)$. \vspace{2mm}
    
    \item\label{H2} Closure under composition with smooth functions:
    
    \noindent If $r\in\R^+$, $u \in C^r(\R^n)$ and $F\in C^\infty(\R)$, then $F(u) \in C^r(\R^n)$ as well. \vspace{2mm}
    
     \item\label{H3} Algebra property:
    
    \noindent If $r\in\R^+\setminus\mathbb Z^+$ and $u,v\in C^r(\R^n)$, then $uv\in C^r(\R^n)$ as well. \vspace{2mm}
    
    \item\label{H4} Mapping property for pseudodifferential operators:
    
    \noindent If $r\in\R^+\setminus\mathbb Z^+$ and $m\in \R$ verify $r-m\in\R^+\setminus\mathbb Z^+$, and $\Psi\in OPS^m_{1,0}$ is a $\Psi$DO with symbol in the H\"ormander class $S^m_{1,0}$, then $\Psi$ maps $C^r(\R^n)$ continuously to $C^{r-m}(\R^n)$. \vspace{2mm}
    
    \item\label{H5} Mapping property for $(-\Delta)^{s-1}\partial_i\partial_j$:
    
    \noindent If $r\in\R^+\setminus\mathbb Z^+$ and $s\in (0,1)$ verify $r-2s\in\R^+\setminus\mathbb Z^+$, then $(-\Delta)^{s-1}\partial_i\partial_j$ maps $C^r(\R^n)$ continuously to $C^{r-2s}(\R^n)$. 
\end{enumerate}

We will also make use of the following Lemma:

\begin{lemma}\label{holder-sobo-lemma}
Let $s\in(0,1)$ and $\eps>0$. If $h\in C^{s+\eps}(\R^n)$ and $u\in H^s(\R^n)$, then $hu\in H^s(\R^n)$ as well, and there exists a constant $C=C(n,s,\eps)>0$ such that $$ \|hu\|_{H^s(\R^n)} \leq C \|h\|_{C^{s+\eps}(\R^n)}\|u\|_{H^s(\R^n)}. $$
\end{lemma}

\begin{proof}
Let $u,h$ be as in the statement. Then we know that
$$ \|hu\|^2_{H^s(\R^n)}\lesssim \|hu\|^2_{L^2(\R^n)} + [hu]^2_{H^s(\R^n)}, $$
(see e.g. \cite{DPV11}) where the Gagliardo seminorm is given by
$$ [hu]^2_{H^s(\R^n)}:= \int_{\R^n}\int_{\R^n} \frac{|h(x)u(x)-h(y)u(y)|^2}{|x-y|^{n+2s}}dxdy. $$
Since $h\in L^\infty(\R^n)$, the first term on the right hand side is immediately estimated as $$ \|hu\|_{L^2(\R^n)}\leq \|h\|_{L^\infty(\R^n)}\|u\|_{L^2(\R^n)}\leq \|h\|_{L^\infty(\R^n)}\|u\|_{H^s(\R^n)}. $$
For the seminorm we compute
\begin{align*}
    [hu]^2_{H^s(\R^n)} & \lesssim \int_{\R^n}\int_{\R^n} \frac{|h(x)|^2|u(x)-u(y)|^2}{|x-y|^{n+2s}}dxdy+ \int_{\R^n}\int_{\R^n} \frac{|h(x)-h(y)|^2|u(y)|^2}{|x-y|^{n+2s}}dxdy =: I_1+I_2,
\end{align*}
and then observe that 
\begin{align*}
     I_1 & \leq \|h\|_{L^\infty(\R^n)}^2[u]^2_{H^s(\R^n)} \leq \|h\|_{L^\infty(\R^n)}^2\|u\|^2_{H^s(\R^n)},
\end{align*}
so that we only need to estimate $I_2$. With a change of variables we get
\begin{align*}
    I_2 & = \int_{\R^n}|u(y)|^2\int_{\R^n} \frac{|h(y+z)-h(y)|^2}{|z|^{n+2s}}dzdy,
\end{align*}
and then we split the second integral as follows, where $B$ is the unitary ball centered at the origin:
\begin{align*}
    \int_{\R^n} \frac{|h(y+z)-h(y)|^2}{|z|^{n+2s}}dz & = \int_{B} \frac{|h(y+z)-h(y)|^2}{|z|^{n+2s}}dz + \int_{\R^n\setminus B} \frac{|h(y+z)-h(y)|^2}{|z|^{n+2s}}dz \\ & \lesssim \|h\|_{C^{s+\eps}(\R^n)}^2 \int_{B} \frac{|z|^{2s+2\eps}}{|z|^{n+2s}}dz + \|h\|^2_{L^\infty(\R^n)}\int_{\R^n\setminus B} \frac{1}{|z|^{n+2s}}dz.
\end{align*}
Observe that the remaining integrals converge to constant values depending on $n,s,\eps$. Thus
$$ I_2 \lesssim \|h\|_{C^{s+\eps}(\R^n)}^2\int_{\R^n}|u(y)|^2dy = \|h\|_{C^{s+\eps}(\R^n)}^2\|u\|^2_{L^2(\R^n)}\leq \|h\|_{C^{s+\eps}(\R^n)}^2\|u\|^2_{H^s(\R^n)}. $$
\end{proof}

\subsection{Spaces of tensors}\label{subsec:sequences}
If $A=A_{\alpha_{1}, ... , \alpha_{m}}$ with $\alpha_i \in \{1, ..., a_i\}$ for $i=1, ..., m$ and $a\in \mathbb N_{\infty}^m$ for $m\in\N$ is any (extended) tensor whose elements belong to $C^r$, $r\in\R$, we define its $C^r$ norm as $$\|A\|_{C^r}:=\sum_{\alpha_1=1}^{a_1}... \sum_{\alpha_m=1}^{a_m} \|A_{\alpha_{1}, ... , \alpha_{m}}\|_{C^r}.$$
Similarly, if $B=B_{\beta_{1}, ... , \beta_{n}}$ with $\beta_j \in \{1, ..., b_j\}$ for $j=1, ..., n$ and $b\in \mathbb N_{\infty}^n$ for $n\in\N$ is any (extended) tensor whose elements belong to $H^r$, we define its $H^r$ norm as 
$$\|B\|_{H^r}:=\sum_{\beta_1=1}^{b_1}... \sum_{\beta_n=1}^{b_n} \|B_{\beta_{1}, ... , \beta_{n}}\|_{H^r}.$$
The result of Lemma \ref{holder-sobo-lemma} extends to products of tensors: if $A\in C^{r+\eps}$ and $B\in H^r$, then $A \otimes B, A\odot B$ and $A\cdot_k B$ all belong to $H^r$ whenever the operations make sense, since \begin{equation}
    \begin{split}
        \|A\otimes B\|_{H^r} & = \sum_{\alpha_1,..., \alpha_m,\beta_1,...,\beta_n}\| A_{\alpha_{1}, ... , \alpha_{m}}B_{\beta_{1}, ... , \beta_{n}} \|_{H^r} \\ & \lesssim  \sum_{\alpha_1,..., \alpha_m,\beta_1,...,\beta_n}\| A_{\alpha_{1}, ... , \alpha_{m}}\|_{C^{r+\eps}}\|B_{\beta_{1}, ... , \beta_{n}} \|_{H^r} \\ & =  \sum_{\alpha_1,..., \alpha_m}\| A_{\alpha_{1}, ... , \alpha_{m}}\|_{C^{r+\eps}} \sum_{\beta_1,...,\beta_n}\|B_{\beta_{1}, ... , \beta_{n}} \|_{H^r}  = \|A\|_{C^{r+\eps}}\|B\|_{H^r},
    \end{split}
\end{equation}
\begin{equation}
    \begin{split}
        \|A\odot B\|_{H^r} & =  \sum_{\alpha_1,..., \alpha_m}\| A_{\alpha_{1}, ... , \alpha_{m}}B_{\alpha_{1}, ... , \alpha_{m}} \|_{H^r}  \\ & \lesssim  \sum_{\alpha_1,..., \alpha_m}\| A_{\alpha_{1}, ... , \alpha_{m}}\|_{C^{r+\eps}}\|B_{\alpha_{1}, ... , \alpha_{m}} \|_{H^r} \\ & \leq \sum_{\alpha_1,..., \alpha_m}\| A_{\alpha_{1}, ... , \alpha_{m}}\|_{C^{r+\eps}} \sum_{\beta_1,...,\beta_n}\|B_{\beta_{1}, ... , \beta_{n}} \|_{H^r}  = \|A\|_{C^{r+\eps}}\|B\|_{H^r},
    \end{split}
\end{equation}
and
\begin{equation}
    \begin{split}
        \|A\cdot_k B\|_{H^r} & =  \sum_{\alpha_1,..., \alpha_{m-k},\beta_{k+1},...,\beta_n}\| \sum_{\beta_1,...,\beta_k}A_{\alpha_1, ..., \alpha_{m-k}, \beta_1, ..., \beta_k} B_{\beta_1, ..., \beta_n} \|_{H^r} \\ & \leq  \sum_{\alpha_1,..., \alpha_{m-k},\beta_{1},...,\beta_n}\left\| A_{\alpha_1, ..., \alpha_{m-k}, \beta_1, ..., \beta_k} B_{\beta_1, ..., \beta_n} \right\|_{H^r}  \\ & \lesssim \sum_{\alpha_1,..., \alpha_{m-k},\beta_{1},...,\beta_n}\left\| A_{\alpha_1, ..., \alpha_{m-k}, \beta_1, ..., \beta_k}\|_{C^{r+\eps}} \|B_{\beta_1, ..., \beta_n} \right\|_{H^r} \\& \leq \sum_{\alpha_1,..., \alpha_m}\| A_{\alpha_{1}, ... , \alpha_{m}}\|_{C^{r+\eps}} \sum_{\beta_1,...,\beta_n}\|B_{\beta_{1}, ... , \beta_{n}} \|_{H^r}  = \|A\|_{C^{r+\eps}}\|B\|_{H^r}.
    \end{split}
\end{equation}

Moreover, if $A\neq 0$ we have $A/|A|^2\in C^{r+\eps}(\R^n)$ by \ref{H2}, since the function $F(x)=x/|x|^2$ is smooth away from the origin. Therefore by the above computation 
\begin{equation}
    \begin{split}
        \|B\|_{H^r} = \left\|\frac{A}{|A|^2}\cdot_m (A\otimes B)\right\|_{H^r} & \lesssim \left\|\frac{A}{|A|^2}\right\|_{C^{r+\eps}}\left\|A\otimes B\right\|_{H^r}.
    \end{split}
\end{equation}
It now easily follows that if $A\in C^{r+\eps}$, then $A\otimes H^r$ is a Hilbert subspace of $H^r$.

\begin{remark}
We might as well have chosen to define 
\begin{equation}\label{alternative-definition}
    \|A\|^2_{C^r}:=\sum_{\alpha_1=1}^{a_1}... \sum_{\alpha_m=1}^{a_m} \|A_{\alpha_{1}, ... , \alpha_{m}}\|^2_{C^r},
\end{equation}
and similarly for the $H^r$ norm of a tensor. If the tensors are not extended, i.e. if $a\in \mathbb N^m$ rather than $\mathbb N_{\infty}^m$, the two definitions of course coincide in light of the Cauchy-Schwarz inequality for series, since
$$ \sum_{\alpha_1=1}^{a_1}... \sum_{\alpha_m=1}^{a_m} \|A_{\alpha_{1}, ... , \alpha_{m}}\|_{C^r} \leq \left(\sum_{\alpha_1=1}^{a_1}... \sum_{\alpha_m=1}^{a_m} \|A_{\alpha_{1}, ... , \alpha_{m}}\|^2_{C^r}\right)^{1/2}\prod_{j=1}^m a_j^{1/2} $$
and
$$ \sum_{\alpha_1=1}^{a_1}... \sum_{\alpha_m=1}^{a_m} \|A_{\alpha_{1}, ... , \alpha_{m}}\|^2_{C^r} \leq \left(\sum_{\alpha_1=1}^{a_1}... \sum_{\alpha_m=1}^{a_m} \|A_{\alpha_{1}, ... , \alpha_{m}}\|_{C^r}\right)^2. $$
However, since the first inequality does not hold for extended tensors, the two definitions are not equivalent for $a\in\mathbb N_{\infty}^m$. In particular, definition \eqref{alternative-definition} would not allow us to get the desired estimate for $\|A\cdot_k B\|_{H^r}$.
\end{remark}

\section{The anisotropic fractional conductivity equation}\label{sec-oper}

\subsection{Assumptions}\label{subsec-ass}
Any square matrix $A$ depending on two variables $x,y$ is linked to two different concepts of symmetry. One can define the \emph{matrix-wise symmetric part} $A_{ms}$ of $A$ as
$$ A_{ms}(x,y) := \frac{A(x,y) + A^T(x,y)}{2}, $$
where $T$ indicates transposition, and the \emph{variable-wise symmetric part} $A_{vs}$ of $A$ as
$$ A_{vs}(x,y) := \frac{A(x,y) + A(y,x)}{2}. $$
One sees that $A_{ms}$ is a symmetric matrix in the sense that $A_{ms} = A_{ms}^T$, while $A_{vs}$ is a symmetric function of the variables $x,y$ in the sense that $A_{vs}(x,y)= A_{vs}(y,x)$. We also define the  \emph{(matrix-wise and variable-wise) antisymmetric parts} of $A$ as $$ A_{ma} := A-A_{ms}, \qquad A_{va} := A-A_{vs}. $$
Thus we can write $$A = A_{vs} + A_{va} = (A_{vs})_{ms} + (A_{vs})_{ma} + A_{va} = A_s + A_a,$$ where $A_s := (A_{vs})_{ms}$ is symmetric both matrix-wise and variable-wise, and $A_a := (A_{vs})_{ma} + A_{va}$.\\

Let now $s\in(0,1)$ and $\Omega\subset\R^n$ be bounded and open. Moreover, let $A\in L^\infty(\R^{2n}, \R^{n\times n})$ verify the following assumptions:

\begin{enumerate}[label=(A\arabic*)]
    \item\label{ass-exterior} There exist two families of matrix-wise symmetric functions $\{\beta_k\}_{k\in\N}\subset L^\infty(\R^n,\R^{n\times n})$, each constant in $\Omega_e$, and $\{\phi_k\}_{k\in\N}\subset L^2(\Omega,\C^{n\times n})$, each compactly supported in $\Omega$, such that for almost every $(x,y)\in\R^{2n}$ \begin{equation}\label{first-decomp}A_s(x,y)= \tilde a(x,y) + \sum_{k\in\N} \phi_k(x)\odot\phi_k(y),\end{equation}
    where
    $$ \tilde a(x,y):= \sum_{k\in\N} \beta_k(x)\odot\beta_k(y) $$
    coincides with $A_s(x,y)$ whenever $(x,y)\not\in\Omega^2$. We also define a third family of matrix-wise symmetric functions $\{\Phi_k\}_{k\in\N}$ given by
    $$ \Phi_{2k}:= \beta_k, \qquad \Phi_{2k+1}:=\phi_k \qquad\qquad\mbox{for all } k\in\N, $$
    so that we can write
    \begin{equation}\label{decomposition}
    A_s(x,y) = \sum_{k\in\N} \Phi_k(x) \odot \Phi_k(y).
\end{equation}
    \item\label{ass-positive} $A_s$ is uniformly positive definite, i.e. there exists a constant $\nu>0$ such that $$A_s(x,y) :(\xi\otimes \xi) \geq \nu |\xi|^2,\qquad \mbox{ for all } \xi\in \R^{n} \mbox{ and } x,y\in\R^n.$$
    \item\label{ass-Phi} There exists $\eps > 0$ such that $\Phi\in C^{2s+\eps}(\R^n)$.
\end{enumerate}


In the next section we will use these assumptions in order to construct the anisotropic fractional conductivity operator having $A$ as coefficient matrix. The rest of this section is dedicated to comments about the assumptions.

\begin{itemize}
    \item Assumption \ref{ass-Phi} implies that the functions $\beta_k$ and $\phi_k$ belong to $C^{2s+\eps}(\R^n)$ for all $k\in\N$. It also gives $A_s\in C^{2s+\eps}(\R^{2n})$, since
\begin{align*}
    \|A_s\|_{C^{2s+\eps}(\R^{2n})} & \leq \sum_{k\in\N}\| \Phi_k(x)\odot\Phi_k(y)  \|_{C^{2s+\eps}(\R^{2n})} \\ & = \sum_{k\in\N}\sum_{i,j=1}^n\| \Phi_{k,ij}(x)\Phi_{k,ij}(y)  \|_{C^{2s+\eps}(\R^{2n})}  \\ & \lesssim \sum_{k\in\N}\sum_{i,j=1}^n \|\Phi_{k,ij}\|_{C^{2s+\eps}(\R^{2n})}^2 \\ & \lesssim \sum_{k\in\N}\sum_{i,j=1}^n \|\Phi_{k,ij}\|_{C^{2s+\eps}(\R^n)}^2 \leq \|\Phi\|_{C^{2s+\eps}(\R^n)}^2,
\end{align*} where at the third step we used property \ref{H3}, and at the following one just the definition of H\"older norm. Similarly, $\tilde a\in C^{2s+\eps}(\R^{2n})$.
    \item The particular structure of the exterior values $\tilde a$ of $A_s$ emerges naturally as a generalization of the exterior condition for the isotropic fractional conductivity equation studied in \cite{C20}. In that case it holds  $$A(x,y) = A_s(x,y) = \gamma^{1/2}(x)\gamma^{1/2}(y) Id = (\gamma^{1/2}(x)Id) \odot (\gamma^{1/2}(y)Id),$$
and thus it is sufficient to take $\tilde a = A$. 
    \item The existence of a sequence $\{\phi_k\}_{k\in\N}$ as in \ref{ass-exterior} such that \eqref{first-decomp} holds can be proved in the assumption that the scalar function $(A_s-\tilde a)_{ij}$ is compactly supported in $\Omega^2$. Fix $i,j\in\{1,...,n\}$, and let $\psi:= (A_s-\tilde a)_{ij}$. In this case we define the Hilbert-Schmidt integral operator $\Psi: L^2(\Omega) \rightarrow L^2(\Omega)$ by
$$ \Psi u(x) := \int_{\Omega} \psi(x,y)u(y) dy. $$ 
Since $\psi$ is real-valued, symmetric in $x,y$, and compactly supported, the operator $\Psi$ is compact and self-adjoint. By the spectral theorem there exists an orthonormal basis $ \{\phi_i\}_{i\in\N}$ of $L^2(\Omega)$ composed of eigenfunctions of $\Psi$ with real eigenvalues $\{\lambda_i\}_{i\in \N}$ such that $|\lambda_1|\geq |\lambda_2|\geq ... \geq 0$. One immediately sees that the new double sequence $\{\phi_i(x)\phi_j(y)\}_{i,j\in\N}$ is an orthonormal basis of $L^2(\Omega^2)$. By writing $\psi$ in this basis, we see that for all $x,y\in\Omega$
\begin{align*}
    \psi(x,y) & = \sum_{i,j} \left(\int_{\Omega\times\Omega}\psi(w,z)\phi_i(w)\phi_j(z) dwdz \right)\phi_i(x)\phi_j(y) \\ & = \sum_{i,j} \left(\int_{\Omega}\phi_i(w)\Psi\phi_j(w) dw \right)\phi_i(x)\phi_j(y) \\ & = \sum_{i,j}\lambda_j  \left(\int_{\Omega}\phi_i(w)\phi_j(w) dw \right)\phi_i(x)\phi_j(y) \\ & = \sum_{i} \lambda_i \phi_i(x)\phi_i(y),
\end{align*}
since we have the equalities $$\Psi\phi_j = \lambda_j\phi_j, \quad\mbox{for all } j\in \N,\qquad \mbox{and}\qquad \int_{\Omega}\phi_i(w)\phi_j(w) dw = \delta_{ij}.$$ If $\phi'_i$ is the extension by $0$ of $\sqrt{\lambda_i}\phi_i$ from $\Omega$ to $\R^n$, then the decomposition
$$ \psi(x,y)=\sum_{i} \phi'_i(x)\phi'_i(y) $$ holds for all $x,y\in \R^n$, and formula \eqref{first-decomp} follows. 
Thus we see that the decomposition given in equation \eqref{first-decomp} is actually an assumption just on the exterior values of $A_s$.

If in particular $\psi$ is a positive semidefinite kernel in $\Omega^2$, i.e. $$ \int_{\Omega}\int_{\Omega}\psi(x,y)g(x)g(y)dxdy \geq 0, \qquad \mbox{for all } g\in L^2(\Omega),$$ then by Mercer's theorem \cite[Section 8.7]{JR82} all the eigenvalues $\lambda_i$ are non-negative,  
$$\mbox{tr}\Psi:=\int_{\Omega}\psi(x,x)dx = \sum_i\lambda_i < \infty,$$
and the functions $\phi'_i$ are real valued.

\item  If $A_1\sim A_2$ as in Definition \ref{def:gauge}, then the decomposition \eqref{decomposition} of $A_{2,s}$ can be taken such that $\Phi_{2,k}(x)=(1+\rho(x))\Phi_{1,k}(x)$ for all $x\in \R^n$ and $k\in\N$. In particular, $\tilde a_1 = \tilde a_2$ on $(\Omega_e)^2$.

\end{itemize}

\subsection{The operator}
For this subsection we do not need any of the assumptions \ref{ass-exterior}-\ref{ass-Phi}. The fact that $A\in L^\infty$, the mapping properties of the fractional gradient $\nabla^s$ and the Cauchy-Schwartz inequality are enough to imply that $A\cdot\nabla^su\in L^2(\R^{2n})$ for all $u\in H^s(\R^n)$. With this in mind, we can define the \emph{anisotropic fractional conductivity operator} $\mathbf C^s_A : H^s(\R^n) \rightarrow H^{-s}(\R^n)$ as 
$$ \mathbf C^s_A u := (\nabla\cdot)^s (A(x,y)\cdot \nabla^s u) $$ for $u\in H^s(\R^n)$. If there exists $\sigma: \R^{2n}\rightarrow \R$ such that $A(x,y) = \sigma(x,y)Id$, the fractional conductivity operator is said to be \emph{isotropic}, and
$$ \mathbf C^s_{\sigma Id} u= (\nabla\cdot)^s(\sigma(x,y) \nabla^s u). $$
If the two-points function $\sigma$ admits the decomposition $\sigma(x,y)=\gamma^{1/2}(x)\gamma^{1/2}(y)$ for some function $\gamma: \R^n \rightarrow \R$, then we recover the usual fractional conductivity operator $\mathbf C^s_\gamma$ studied in \cite{C20}. Finally, in the case $\gamma(x) \equiv 1$ our operator reduces to the fractional Laplacian $(-\Delta)^s$.

\begin{remark}\label{natural-gauge}For any $u,v\in H^s$ we see that
\begin{align*}
    \langle \mathbf C^s_Au, v \rangle & = \langle A, \nabla^sv \otimes \nabla^s u\rangle \\ & = \langle A_s, \nabla^sv \otimes \nabla^s u\rangle + \langle (A_{vs})_{ma}, \nabla^sv \otimes \nabla^s u\rangle + \langle A_{va}, \nabla^sv \otimes \nabla^s u\rangle .
\end{align*}
Observe that the matrix $\nabla^sv\otimes\nabla^su$ is symmetric both matrix- and variable-wise. This makes the last two terms on the right hand side vanish, and we are left with
\begin{equation}\label{symmetry}
     \langle \mathbf C^s_Au, v \rangle =  \langle A_s, \nabla^sv \otimes \nabla^s u\rangle = \langle \mathbf C^s_{A_s}u, v \rangle .
\end{equation}
This means that the operator $\mathbf C^s_A$ actually contains no information about $A_a$, which is therefore unrecoverable. Thus the anisotropic fractional conductivity equation possesses a natural gauge. This is reminiscent of the situation emerging for the fractional magnetic Schr\"odinger equation studied in \cite{C20a}. 
\end{remark}

\begin{lemma}
The operator $\mathbf C^s_A$ is self-adjoint.
\end{lemma}

\begin{proof}
By Remark \ref{natural-gauge} we have
$$ \langle \mathbf C^s_Au, v \rangle = \langle \mathbf C^s_{A_s}u, v \rangle = \langle u,\mathbf C^s_{A_s}v \rangle = \langle u,\mathbf C^s_Av \rangle ,  $$
since the symmetries of $A_s$ ensure that $C^s_{A_s}$ is itself self-adjoint.
\end{proof}

\subsection{The reduction lemma}

The next result is a reduction lemma, by means of which we can write $\mathbf C^s_A$ as a combination of fractional Laplacian operators acting on sequences. This result is related to the fractional Liouville reduction studied in \cite{C20} for the isotropic, separable case, and it also resembles the reduction lemma presented in \cite{CdHS22} for the fractional elasticity equation.

\begin{lemma}\label{reduction-lemma}
In weak sense it holds that
$$ \mathbf C^s_A u = \Phi \tridots \left\{(-\Delta)^{s-1}D(\Phi u) + (\Phi u)\tridots  Q\right\},$$
where $D:= - \frac{\left(\Delta Id + 2s\nabla^2\right)\odot}{n+2s}$ is a second order operator acting between $n\times n$ matrices and $Q:=- \frac{\Phi\otimes (-\Delta)^{s-1}D\Phi}{|\Phi|^2}$ is the transformed potential.
\end{lemma}

\begin{proof}
Let $u,v\in H^s$, and compute
\begin{align*}
    \langle\mathbf C^s_A u ,v\rangle& = \langle A_s\cdot\nabla^s u, \nabla^sv\rangle \\ & = \langle \sum_{k\in \N} (\Phi_k(x)\odot\Phi_k(y)) \cdot \nabla^s u, \nabla^sv \rangle \\ & = \sum_{k\in \N}\langle  (\Phi_k(x)\odot\Phi_k(y))\cdot[(u(y)-u(x))\otimes \zeta],(v(y)-v(x))\otimes \zeta \rangle \\ & = \sum_{k\in\N}\langle (u(y)-u(x))(\Phi_k(x) \odot \Phi_k(y)):(\zeta\otimes\zeta) , v(y)-v(x) \rangle .
\end{align*}
Observe that the order of summation and integration can be exchanged, given that for all $k\in\N$ and almost all $x,y\in\R^n$
\begin{align*}
    |(\Phi_k(x)\odot\Phi_k(y)):(\nabla^sv\otimes\nabla^su)|& \leq |\Phi_k(x)\odot\Phi_k(y)|\,|\nabla^sv\otimes\nabla^su| \\ & \leq |\Phi_k(x)|\,|\Phi_k(y)|\,|\nabla^su|\,|\nabla^sv| \\ & \leq \|\Phi_k\|_{L^\infty(\R^n)}^2|\nabla^su|\,|\nabla^sv| \\& \leq \|\Phi_k\|_{C^{2s+\eps}(\R^n)}^2|\nabla^su|\,|\nabla^sv| \\& \leq \|\Phi\|_{C^{2s+\eps}(\R^n)}^2|\nabla^su|\,|\nabla^sv| =: g(x,y),
\end{align*} 
and $g\in L^1(\R^{2n})$. For constants $\alpha= (n+2s)^{-1}$ and $\beta = (n+2s)^{-1}(n+2s-2)^{-1}$ we have (see e.g. \cite{CdHS22}) $$ \frac{(x-y)\otimes(x-y)}{|x-y|^{n+2s+2}} = \alpha Id |x-y|^{-(n+2s)} - \beta \nabla_y\nabla_x(|x-y|^{-(n+2s-2)}), $$
and thus $$\zeta\otimes\zeta = \frac{C_{n,s}}{2}\left( \alpha Id |x-y|^{-(n+2s)} - \beta \nabla_y\nabla_x(|x-y|^{-(n+2s-2)}) \right).$$
It follows that we can write $\langle \mathbf C^s_A u,v\rangle = \sum_k\left(\alpha I_{1,k} -\frac{C_{n,s}}{2}\beta I_{2,k}\right)$, with
\begin{align*}
I_{1,k} &:=  \frac{C_{n,s}}{2}\langle Id:(\Phi_k(x) \odot \Phi_k(y))(u(y)-u(x))|x-y|^{-(n+2s)} , v(y)-v(x) \rangle,    
\end{align*}
and
$$ I_{2,k} := \langle (u(y)-u(x))(\Phi_k(x) \odot \Phi_k(y)):\nabla_y\nabla_x(|x-y|^{-(n+2s-2)}) , v(y)-v(x) \rangle. $$
For $I_{1,k}$ we compute
\begin{align*}
I_{1,k} &= \sum_i \langle \Phi_{k,i,i}(x) \Phi_{k,i,i}(y)\nabla^su , \nabla^sv \rangle \\ & = \sum_i \langle \Phi_{k,i,i}(-\Delta)^s(\Phi_{k,i,i}u) - \Phi_{k,i,i} u (-\Delta)^s\Phi_{k,i,i} , v \rangle \\ & = Id: \langle \Phi_{k}\odot(-\Delta)^s(\Phi_{k}u) - (\Phi_{k} u)\odot (-\Delta)^s\Phi_{k} , v \rangle ,    
\end{align*}
where at the second line we recognized the scalar fractional conductivity operator studied in \cite{C20} with conductivity $\Phi_{k,i,i}^2$, and thus applied the usual fractional Liouville reduction. We want to compute the second term $I_{2,k}$ integrating by parts. We let $w_k:=\Phi_ku$ and compute 
\begin{align*}
    I_{2,k} &= \int_{\R^{2n}} (\Phi_k(x)\odot w_k(y)-w_k(x)\odot\Phi_k(y))(v(y)-v(x)):\nabla_y\nabla_x(|x-y|^{-(n+2s-2)})dydx \\ & = -\int_{\R^{2n}} (\nabla_y\cdot)\left\{ (\Phi_k(x)\odot w_k(y)-w_k(x)\odot\Phi_k(y))(v(y)-v(x)) \right\} \cdot\nabla_x( |x-y|^{-(n+2s-2)}) dydx \\ & = -\int_{\R^{2n}} \left\{ (v(y)-v(x))\nabla_y\cdot [ (\Phi_k(x)\odot w_k(y)-w_k(x)\odot\Phi_k(y))] +\right. \\ & \qquad\qquad\qquad + \left.\nabla v(y)\cdot [ (\Phi_k(x)\odot w_k(y)-w_k(x)\odot\Phi_k(y))] \right\} \cdot\nabla_x( |x-y|^{-(n+2s-2)}) dydx \\ & = \int_{\R^{2n}} (\nabla_x\cdot)\left\{ (v(y)-v(x))(\nabla_y\cdot) [ (\Phi_k(x)\odot w_k(y)-w_k(x)\odot\Phi_k(y))] +\right. \\ & \qquad\qquad\qquad + \left.\nabla v(y)\cdot [ (\Phi_k(x)\odot w_k(y)-w_k(x)\odot\Phi_k(y))] \right\}  |x-y|^{-(n+2s-2)} dydx \\ & = \int_{\R^{2n}} \left\{ (v(y)-v(x))(\nabla_x\cdot)(\nabla_y\cdot) [ (\Phi_k(x)\odot w_k(y)-w_k(x)\odot\Phi_k(y))] +\right. \\ & \qquad\qquad\qquad - \nabla v(x) \cdot (\nabla_y\cdot) [ (\Phi_k(x)\odot w_k(y)-w_k(x)\odot\Phi_k(y))] + \\ & \qquad\qquad\qquad + \left.\nabla v(y)\cdot (\nabla_x\cdot)[ (\Phi_k(x)\odot w_k(y)-w_k(x)\odot\Phi_k(y))] \right\}  |x-y|^{-(n+2s-2)} dydx.
\end{align*}
Now observe that
$$ \{(\nabla_y\cdot)(A(x)\odot B(y))\}_j = \sum_i \partial_{y_i} (A_{ij}(x)B_{ij}(y)) = \sum_i A_{ij}(x)\partial_{y_i}B_{ij}(y) $$
and similarly
$$ (\nabla_x\cdot)(\nabla_y\cdot)(A(x)\odot B(y)) = \sum_j \partial_{x_j}\{(\nabla_y\cdot)(A(x)\odot B(y))\}_j = \sum_{i,j} \partial_{x_j}A_{ij}(x)\partial_{y_i}B_{ij}(y).  $$
This implies 
\begin{align*}
    I_{2,k}&= \sum_{i,j}\int_{\R^{2n}}  v(y) \{ \partial_{j}\Phi_{k,ij}(x)\partial_{i}w_{k,ij}(y) - \partial_{j}w_{k,ij}(x)\partial_{i}\Phi_{k,ij}(y) \} |x-y|^{-(n+2s-2)}dydx \\ & \quad + \sum_{i,j} \int_{\R^{2n}} \partial_j v(y)  \{ \partial_i \Phi_{k,ij}(x)w_{k,ij}(y) - \partial_i w_{k,ij}(x)\Phi_{k,ij}(y) \} |x-y|^{-(n+2s-2)}dydx \\ & \quad - \sum_{i,j}\int_{\R^{2n}}  v(x) \{ \partial_{j}\Phi_{k,ij}(x)\partial_{i}w_{k,ij}(y) - \partial_{j}w_{k,ij}(x)\partial_{i}\Phi_{k,ij}(y) \} |x-y|^{-(n+2s-2)}dydx \\ & \quad - \sum_{i,j} \int_{\R^{2n}} \partial_j v(x)  \{ \Phi_{k,ij}(x)\partial_{i}w_{k,ij}(y) - w_{k,ij}(x)\partial_{i}\Phi_{k,ij}(y)\} |x-y|^{-(n+2s-2)}dydx.
\end{align*}

Now we exchange $x$ and $y$ in the last two integrals, observing that the $n+2s-2 < n$ exponent ensures the well-definiteness of all the integrals involved. At this step we also use that $\Phi_k$ and $w_k$ are symmetric matrices for all values of $k$. We get
\begin{align*}
    \frac{I_{2,k}}{2}&= \sum_{i,j}\int_{\R^{2n}}  v(y) \{ \partial_{j}\Phi_{k,ij}(x)\partial_{i}w_{k,ij}(y) - \partial_{j}w_{k,ij}(x)\partial_{i}\Phi_{k,ij}(y) \} |x-y|^{-(n+2s-2)}dydx \\ & \quad + \sum_{i,j} \int_{\R^{2n}} \partial_j v(y)  \{ \partial_i \Phi_{k,ij}(x)w_{k,ij}(y) - \partial_i w_{k,ij}(x)\Phi_{k,ij}(y) \} |x-y|^{-(n+2s-2)}dydx.
\end{align*}
Let us rewrite the last formula equivalently as
\begin{align*}
    \frac{C_{n,s-1}I_{2,k}}{2}&= \sum_{i,j}\int_{\R^{n}}  v  \{ R\partial_{j}\Phi_{k,ij}\, \partial_{i}w_{k,ij} - R\partial_{j}w_{k,ij}\,\partial_{i}\Phi_{k,ij} \} dy \\ & \quad + \sum_{i,j} \int_{\R^{n}} \partial_i v  \{ R\partial_j \Phi_{k,ij}\,w_{k,ij} - R\partial_j w_{k,ij}\,\Phi_{k,ij} \} dy,
\end{align*}
where for the sake of readability we let $R:=(-\Delta)^{s-1}$. Since $R$ commutes with the derivatives, one last integration by parts in the last integral gives
\begin{align*}
    \int_{\R^{n}} \partial_i v  \{ R\partial_j \Phi_{k,ij}\,w_{k,ij} - R\partial_j w_{k,ij}\,\Phi_{k,ij} \} dy & = - \int_{\R^{n}} v  \partial_i\{ R\partial_j \Phi_{k,ij}\,w_{k,ij} - R\partial_j w_{k,ij}\,\Phi_{k,ij} \} dy \\ & =
    -\int_{\R^n} v\{ w_{k,ij} \, R\partial_i\partial_j \Phi_{k,ij} - \Phi_{k,ij}\,R\partial_i\partial_j w_{k,ij}\}dy \\ & \quad - \int_{\R^n} v\{ \partial_iw_{k,ij} \, R\partial_j \Phi_{k,ij} - \partial_i\Phi_{k,ij}\,R\partial_j w_{k,ij}\}dy,
\end{align*}
and so eventually
\begin{align*}
    \frac{C_{n,s-1}I_{2,k}}{2}&= \sum_{i,j}\int_{\R^n} v\{ \Phi_{k,ij}\,R\partial_i\partial_j w_{k,ij}-w_{k,ij} \, R\partial_i\partial_j \Phi_{k,ij}\}dy \\ & = \langle \Phi_k: (-\Delta)^{s-1}(\nabla^2\odot w_k) - w_k : (-\Delta)^{s-1}(\nabla^2\odot \Phi_k) , v\rangle.
\end{align*}
Putting together the computations for $I_{1,k}$ and $I_{2,k}$ we get the following equality

\begin{align*}
    \mathbf \langle \mathbf C^s_A u,v\rangle & = \sum_{k\in\N}\left(\alpha I_{1,k} -\frac{C_{n,s}}{2}\beta I_{2,k}\right) \\ & = \alpha\sum_{k\in\N} Id: \langle \Phi_{k}\odot(-\Delta)^sw_k - w_k\odot (-\Delta)^s\Phi_{k} , v \rangle \\ &  \quad +\frac{C_{n,s} \beta}{C_{n,s-1}}\sum_{k\in\N} \langle w_k : (-\Delta)^{s-1}(\nabla^2\odot \Phi_k)-\Phi_k: (-\Delta)^{s-1}(\nabla^2\odot w_k), v  \rangle \\ & = \sum_{k\in\N}  \langle \Phi_k : (-\Delta)^{s-1}D w_k -  w_k : (-\Delta)^{s-1}D \Phi_k,v\rangle \\ & = \lim_{K\rightarrow \infty} \langle \sum_{k=1}^K(\Phi_k : (-\Delta)^{s-1}D w_k -  w_k : (-\Delta)^{s-1}D \Phi_k),v\rangle =: \lim_{K\rightarrow \infty} F_K(v) ,
\end{align*}
where the differential operator with constant coefficients $D$ is $$D:= -\left(\alpha\Delta Id + \beta\frac{C_{n,s}}{C_{n,s-1}}\nabla^2\right)\odot =- \frac{\left(\Delta Id + 2s\nabla^2\right)\odot}{n+2s}.$$
By Lemma \ref{holder-sobo-lemma} and property \ref{H1}, $\Phi_k$ is a multiplier on $H^s(\R^n)$ (and thus also in $H^{-s}(\R^n)$, see \cite{MS09, CMRU22}) for every $k\in\N$. Moreover, $(-\Delta)^{s-1}D\Phi_k\in  C^{\eps}(\R^n)$ by property \ref{H5}. This ensures that the operators $F_K$ belong to $H^{-s}(\R^n)$ for all $K\in\N$. They converge to $F:= \Phi \tridots (-\Delta)^{s-1}D(\Phi u) - \Phi u\tridots (-\Delta)^{s-1}D\Phi $ in the norm of $H^{-s}$ as $K\rightarrow \infty$, since
\begin{align*}
    \|F_K-F\|_{H^{-s}} & \leq \|\sum_{k=1}^K(\Phi_k : (-\Delta)^{s-1}D (\Phi_ku))-\Phi \tridots (-\Delta)^{s-1}D(\Phi u)\|_{H^{-s}} \\ & \quad + \|\sum_{k=1}^K(\Phi_ku : (-\Delta)^{s-1}D \Phi_k)-\Phi u\tridots (-\Delta)^{s-1}D\Phi\|_{H^{-s}} \\ & = \|\sum_{k>K}\Phi_k : (-\Delta)^{s-1}D (\Phi_ku)\|_{H^{-s}} + \|\sum_{k>K}\Phi_ku : (-\Delta)^{s-1}D \Phi_k\|_{H^{-s}} \\ & \leq \sum_{k>K}\left(\|\Phi_k : (-\Delta)^{s-1}D (\Phi_ku)\|_{H^{-s}} + \|\Phi_ku : (-\Delta)^{s-1}D \Phi_k\|_{H^{-s}}\right) \\ & \lesssim \sum_{k>K}\left(\|\Phi_k\|_{C^{2s+\eps}}\| (-\Delta)^{s-1}D (\Phi_ku)\|_{H^{-s}} + \|\Phi_ku\|_{H^s} \|(-\Delta)^{s-1}D \Phi_k\|_{C^{\eps}}\right) \\ & \lesssim \sum_{k>K}\|\Phi_k\|_{C^{2s+\eps}}\| \Phi_ku\|_{H^{s}} \\ & \lesssim \|u\|_{H^{s}}\sum_{k>K}\|\Phi_k\|^2_{C^{2s+\eps}} \leq \|u\|_{H^{s}}\left(\sum_{k>K}\|\Phi_k\|_{C^{2s+\eps}}\right)^2 ,
\end{align*}
which converges to $0$ as $K\rightarrow\infty$ by assumption \ref{ass-Phi}. Given that 
$$ |F_K(v)-F(v)|\leq \|v\|_{H^s}\sup_{\|\xi\|_{H^s}=1}|F_K(\xi)-F(\xi)| = \|v\|_{H^s}\|F_K-F\|_{H^{-s}}\rightarrow 0,$$
we deduce that $\mathbf C^s_A u= \Phi \tridots (-\Delta)^{s-1}D(\Phi u) - \Phi u\tridots (-\Delta)^{s-1}D\Phi$ holds in weak sense. Equivalently, if the transformed potential $Q$ is defined as $Q:= -\frac{\Phi\otimes (-\Delta)^{s-1}D\Phi}{|\Phi|^2}$, one can write
\begin{align*}
    \mathbf C^s_A u & = \Phi \tridots \left\{(-\Delta)^{s-1}D(\Phi u) - u (-\Delta)^{s-1}D\Phi\right\} \\ & = \Phi \tridots \left\{(-\Delta)^{s-1}D(\Phi u) - \frac{(\Phi u)\tridots \Phi}{|\Phi|^2} (-\Delta)^{s-1}D\Phi\right\} \\ & = \Phi \tridots \left\{(-\Delta)^{s-1}D(\Phi u) + (\Phi u)\tridots  Q\right\}.
\end{align*}
Observe that all the terms on the right hand side make sense in $H^{-s}(\R^n)$. In fact, since $\Phi\in C^{2s+\eps}(\R^n)$ and $|\Phi|>0$ by assumptions \ref{ass-Phi} and \ref{ass-positive}, following the reasoning of Section \ref{subsec:sequences} we have that $\Phi/|\Phi|^2 \in C^{2s+\eps}(\R^n)$. As observed, $(-\Delta)^{s-1}D\Phi\in  C^{\eps}(\R^n)$ by property \ref{H5}. Since by \ref{H1} we have $ C^{2s+\eps}(\R^n)\subset  C^{\eps}(\R^n)$, a new application of \ref{H3} ensures $Q\in C^\eps(\R^n)\subset L^\infty(\R^n)$. Given that $\Phi$ is a multiplier on $H^s(\R^n)$ and $H^{-s}(\R^n)$ by Lemma \ref{holder-sobo-lemma} and property \ref{H1}, we get $(-\Delta)^{s-1}D(\Phi u) + (\Phi u)\tridots Q \in H^{-s}(\R^n)$, and eventually $\Phi \tridots \left\{(-\Delta)^{s-1}D(\Phi u) + (\Phi u)\tridots  Q\right\}\in H^{-s}(\R^n)$. \end{proof}

\subsection{Weakly anisotropic and isotropic matrices} \label{simple-cases}
In this section we look at $\mathbf C^s_A$ for matrices $A$ of special forms. In some of these cases, the structure of the operator simplifies noticeably, and we recover familiar operators related to the fractional Laplacian $(-\Delta)^s$. \vspace{2mm}

Since $A_s(x,y)$ is positive definite and symmetric for all $x,y\in\R^n$, there exist a diagonal matrix $L(x,y)$ with positive entries and an orthonormal matrix $U(x,y)$  such that $$A_s(x,y) = U^T(x,y)\,L(x,y)\,U(x,y)$$
for all $x,y\in\R^n$. The columns of $U$ (that is, the eigenvectors of $A_s$ of unit norm) represent the principal directions of $A_s$ at $(x,y)$, while the entries of $L$ (i.e. the eigenvalues of $A_s$) are the corresponding scalar conductivities in the given directions. The anisotropy of $A_s$ is due to two factors: the principal directions of $A_s$ change according to position, and the $n$ directional conductivities differ among themselves even at any given point. As a result, in general the electric field and the current density are not parallel to each other. \vspace{2mm}

However, there exist materials showing a weaker form of anisotropy, in the sense that the principal directions do not change with respect to position. This is the case of many crystalline materials (such as graphite, high temperature superconductors, and some metals), which are characterized by atomic structures repeating periodically in fixed directions. In this case, the matrices of the family $\{A_s(x,y): x,y\in\R^n\}$ are \emph{simultaneously diagonalizable}, which means that there exists a unique orthonormal matrix $U$ such that the formula
$$ A_s(x,y)= U^T\, L(x,y)\, U $$
holds for all $x,y\in\R^n$. By letting $x':= U\cdot x$, $u':= u\circ U^T $ and $L':= L\circ \left(  \begin{array}{cc}
U^T & 0  \\
0 & U^T  \end{array} \right)$ for all $x\in\R^n$, we see that 
$$ U\cdot\nabla^sv(x,y) =  \frac{C_{n,s}^{1/2}}{\sqrt{2}} \frac{v(y)-v(x)}{|x-y|^{n/2+s+1}}U\cdot(x-y) = \frac{C_{n,s}^{1/2}}{\sqrt{2}} \frac{v'(y')-v'(x')}{|x'-y'|^{n/2+s+1}}(x'-y') = \nabla^sv'(x',y'), $$
and thus the fractional conductivity operator becomes
\begin{align*}
    \langle \mathbf C^s_A u,v\rangle & = \langle A_s(x,y)\cdot\nabla^s u, \nabla^s v \rangle \\ & = \langle L(x,y)\cdot (U\cdot\nabla^su),U\cdot\nabla^sv\rangle \\ & = \langle L'(x',y')\cdot \nabla^su',\nabla^sv' \rangle \\ & = \langle \mathbf C^s_{L'} u', v'\rangle 
\end{align*}

Thus we see that it is interesting to consider $\mathbf C^s_A$ for a diagonal matrix $A$. We shall also assume that the exterior value matrices $\beta_k$ are diagonal for all $k\in\N$, so that the matrices $\phi_k$ can themselves be taken to be diagonal for all $k\in\N$, and $\Phi$ is a sequence of diagonal matrices. Given that the usual matrix product and the Hadamard product coincide for diagonal matrices, in this case we can write
$$ A_s(x,y) = \sum_{k\in\N} \Phi_k(x) \cdot \Phi_k(y).$$

Assume now that the operator is isotropic, that is $A_s(x,y) = \sigma(x,y)$Id for a symmetric scalar function $\sigma$, and $\beta_k = b_k$Id for all $k\in\N$. In this case, the functions $\phi_k^{(i,j)}$ can be taken to vanish whenever $i\neq j$, and there exists a function $\phi_k$ such that $\phi_k^{(i,i)}=\phi_k$ for all $i\in\{1,...,n\}$. In other words, $\Phi_k = \phi_k$Id, and so we have
\begin{align*}
Id : D(u\Phi_k) &= -Id : \frac{\left(\Delta Id + 2s\nabla^2\right)}{n+2s}\odot (u\phi_k Id) = -\Delta (u\phi_k ) .
\end{align*}
This implies that
\begin{align*}
    \mathbf C^s_A u &  = \sum_k  \phi_k Id : (-\Delta)^{s-1}D (u\Phi_k) - u \phi_k Id : (-\Delta)^{s-1}D \Phi_k \\ & = \sum_k  \phi_k \left((-\Delta)^s (u\phi_k ) - u (-\Delta)^s \phi_k\right) \\ &  = \phi \cdot ((-\Delta)^s(u\phi) + (u\phi)\cdot Q),
\end{align*}
where $Q:= -\frac{\phi\otimes (-\Delta)^s\phi}{|\phi|^2}$. We see that in this case the equation retains its form, but it is greatly simplified. If moreover $\sigma$ is a separable function of $x,y$, then $\mathbf C^s_A$ is the usual fractional conductivity operator from \cite{C20}, and Lemma \ref{reduction-lemma} gives the relative fractional Liouville reduction.

\section{Well-posedness of the direct problems and reduction}\label{sec-wellp}

We consider the direct problem
\begin{equation}\label{direct-problem}
    \begin{split}
        \mathbf C^s_{A} u = F & \quad\mbox{ in } \Omega, \\
u = f & \quad\mbox{ in } \Omega_e,
    \end{split}
\end{equation}
for $F \in H^{-s}(\Omega)$ and $f\in H^s(\R^n)$. Define the bilinear form
$$ B^s_A(u,v) := \langle A(x,y)\cdot\nabla^s u, \nabla^s v \rangle = \langle A_s(x,y)\cdot\nabla^s u, \nabla^s v \rangle $$
for $u,v \in C^\infty_c(\R^n)$. The boundedness estimate
\begin{equation}\label{boundedness}
    |B^s_A(u,v)| \leq  \|A\cdot\nabla^s u\|_{L^2(\R^{2n})} \|\nabla^s v\|_{L^2(\R^{2n})} \lesssim \|u\|_{H^s}\|v\|_{H^s}
\end{equation}
holds since $A\in L^\infty$. Thus $B^s_A$ can be extended by density to act as a bounded linear operator on $H^s(\R^n)\times H^s(\R^n)$.
We also define the potential energy $U^s_A$ for all $u\in H^s(\R^n)$ as $$U^s_A(u) := B^s_A(u,u).$$
Since by assumption \ref{ass-positive} $A_s$ is uniformly positive definite, it holds that
$$U^s_A(u) = \langle A_s\cdot\nabla^su,\nabla^su\rangle \geq \nu\|\nabla^su\|_{L^2(\R^{2n})}^2\approx \|(-\Delta)^{s/2}u\|_{L^2(\R^n)}^2,$$
so that if $u\in H^s_K(\R^n)$ for some compact set $K$, by the fractional Poincar\'e inequality (see e.g. \cite{GSU20} and \cite{CdHS22}) we obtain the coercivity estimate
\begin{equation}\label{coercivity}
    B^s_A(u,u) = U^s_A(u) \gtrsim \|u\|_{H^s(\R^n)}^2.
\end{equation}

We say that $u\in H^s(\mathbb R^n)$ is a weak solution to the inhomogeneous problem \eqref{direct-problem} if and only if $B^s_A(u,v)=F(v)$ holds for all $v\in\widetilde H^s(\Omega)$, and $u-f\in\widetilde H^s(\Omega)$. The next Proposition gives the well-posedness for the direct problem: 

\begin{proposition}
Let $s \in (0,1)$ and assume $\Omega \subset \mathbb R^n$ is a bounded open set. For any $f\in H^s(\mathbb R^n)$ and $F\in H^{-s}(\Omega)$ there exists a unique $u\in H^s(\mathbb R^n)$ such that $u-f \in \widetilde H^s(\Omega)$ and
$$ B^s_A(u,v) = F(v) \quad \mbox{for all} \quad v \in \widetilde H^s(\Omega).$$
Moreover, the following estimate holds: 
$$ \|u\|_{H^s(\mathbb R^n)} \leq C\left( \|f\|_{H^s(\mathbb R^n)} + \|F\|_{H^{-s}(\Omega)} \right). $$
\end{proposition}

\begin{proof}
The proof of this statement follows by the Lax-Milgram theorem, given the boundedness and coercivity estimates \eqref{boundedness}, \eqref{coercivity}. It goes along the same lines as the corresponding proposition in \cite{CdHS22}, see also \cite{GSU20}, and it is by now well understood. We thus omit it here. 
\end{proof}

Next, we study the well-posedness of the transformed direct problem
\begin{equation}\label{transformed-problem}
    \begin{split}
        \mathbf (-\Delta)^{s-1}D w + w\tridots Q = G & \quad\mbox{ in } \Omega, \\
w = \Phi f & \quad\mbox{ in } \Omega_e,
    \end{split}
\end{equation}
for $G \in (\Phi \widetilde H^{s}(\Omega))^*$ and $f\in H^s(\R^n)$, where $Q:= -\frac{\Phi\otimes (-\Delta)^{s-1}D\Phi}{|\Phi|^2}$. As for the original direct problem, we define the bilinear form $$B^s_{Q}(w,v) := \langle (-\Delta)^{s-1}Dw, v \rangle + \langle w\tridots Q, v \rangle $$ for $w,v\in \Phi H^s(\R^n)$. We say that $w\in \Phi H^s(\mathbb R^n)$ is a weak solution to the inhomogeneous problem \eqref{transformed-problem} if and only if $B^s_Q(w,v)=G(v)$ holds for all $v\in\Phi \widetilde H^s(\Omega)$, and $w-\Phi f\in \Phi \widetilde H^s(\Omega)$. By virtue of the reduction Lemma \ref{reduction-lemma}, we can compute 
\begin{equation}
    \begin{split}
        \label{equivalence-bilinear-forms}
    B^s_{Q}(\Phi u, \Phi v) & = \langle (-\Delta)^{s-1}D(\Phi u), \Phi v \rangle + \langle (\Phi u)\tridots Q, \Phi v \rangle \\ & = \langle \Phi\tridots \{(-\Delta)^{s-1}D(\Phi u) + (\Phi u)\tridots Q\}, v\rangle \\ & = \langle \mathbf C^s_A u, v\rangle = B^s_A(u,v)
    \end{split}
\end{equation}
for all $u,v\in H^s(\R^n)$, and thus $B^s_Q$ is immediately bounded and coercive in the space $\Phi H^s(\R^n) \times \Phi H^s(\R^n)$ by formulas \eqref{boundedness}, \eqref{coercivity}, and Lemma \ref{holder-sobo-lemma}. Therefore, we obtain the following well-posedness result:
\begin{proposition}
Let $s \in (0,1)$ and assume $\Omega \subset \mathbb R^n$ is a bounded open set. For any $f\in H^s(\mathbb R^n)$ and $G\in (\Phi \widetilde H^{s}(\Omega))^*$ there exists a unique $w\in\Phi H^s(\mathbb R^n)$ such that $w-\Phi f \in \Phi\widetilde H^s(\Omega)$ and
$$ B^s_{Q}(w,v) = G(v) \quad \mbox{for all} \quad v \in \Phi\widetilde H^s(\Omega).$$
\end{proposition}


The next proposition follows immediately from Lemma \ref{reduction-lemma}. 

\begin{proposition}[Fractional Liouville reduction]
Let $s \in (0,1)$ and assume $\Omega \subset \mathbb R^n$ is a bounded open set. If $u\in H^s(\mathbb R^n)$ solves the original problem
\begin{equation}
    \begin{array}{rll}\label{original2}
        \mathbf C^s_A u & =F & \quad \mbox{ in } \Omega \\
       u & = f & \quad \mbox{ in } \Omega_e
    \end{array}
\end{equation}
in weak sense in $\widetilde H^s(\Omega)$ for some $f\in H^s(\mathbb R^n)$ and $F\in H^{-s}(\Omega)$, then $w:=\Phi u$ solves the transformed problem
\begin{equation}
    \begin{array}{rll}\label{easier-one}
       \mathbf (-\Delta)^{s-1}D w + w\tridots Q & = G & \quad\mbox{ in } \Omega \\
       w & = \Phi f & \quad\mbox{ in } \Omega_e
    \end{array}
\end{equation}
in weak sense in $\Phi \widetilde H^s(\Omega)$, where $G:= \frac{\Phi F}{|\Phi|^2}$. Conversely, if $w\in \Phi H^s(\mathbb R^n)$ solves \eqref{easier-one} in weak sense in $\Phi\widetilde H^s(\Omega)$ for some $G\in (\Phi \widetilde H^{s}(\Omega))^*$, then $u:=\frac{\Phi\tridots w}{|\Phi|^2}$ solves \eqref{original2} in weak sense in $\widetilde H^s(\Omega)$, where $F=\Phi\tridots G$.
\end{proposition}

\begin{proof}
Observe that $u\in H^s(\R^n)$ solves \eqref{original2} in weak sense if and only if $u-f\in \widetilde H^s(\Omega)$, and also for all $v\in \widetilde H^s(\Omega)$ we have $B^s_A(u,v)=F(v)$. Then of course $w:=\Phi u$ belongs to $\Phi H^s(\R^n)$, it verifies $w-\Phi f \in \Phi \widetilde H^s(\Omega)$, and also by \eqref{equivalence-bilinear-forms}
$$ B^s_Q(w,\Phi v) = B^s_A(u,v) = \langle F,v\rangle = \langle G, \Phi v\rangle $$
for $G:= \frac{\Phi F}{|\Phi|^2} \in (\Phi \widetilde H^{s}(\Omega))^*$. Conversely, if $w\in \Phi H^s(\mathbb R^n)$ solves \eqref{easier-one} in weak sense in $\Phi\widetilde H^s(\Omega)$ for some $G\in (\Phi \widetilde H^{s}(\Omega))^*$, then $w-\Phi f \in \Phi \widetilde H^s(\Omega)$, and also $ B^s_Q(w,\Phi v) = \langle G, \Phi v\rangle $ for all $v\in \widetilde H^s(\Omega)$. Then $u:= \frac{\Phi \tridots w}{|\Phi|^2}\in H^s(\R^n)$, $u-f = \frac{\Phi\tridots (w-\Phi f)}{|\Phi|^2}\in \widetilde H^s(\Omega)$, and again by \eqref{equivalence-bilinear-forms}
$$ B^s_A(u,v)=B^s_Q(w,\Phi v) = \langle G, \Phi v\rangle = \langle F,v \rangle $$
for $F:= \Phi\tridots G \in H^{-s}(\Omega)$.
\end{proof}

\section{The DN map and the Alessandrini identity}
Following \cite{GSU20}, define the abstract trace space $X:=H^s(\R^n)/\widetilde{H}^s(\Omega)$. Observe that $X=H^s(\Omega_e)$ holds for all Lipschitz $\Omega$. Since problem \eqref{direct-problem} is well-posed, we can define the Poisson operator $P_A: X \rightarrow H^s(\R^n)$ associating to each exterior value $f\in X$ the unique solution $u_f\in H^s(\R^n)$ to \eqref{direct-problem} with $F=0$. It is then possible to define a continuous, self-adjoint, linear map $\Lambda_{A} : X \rightarrow X^*$ by
$$ \langle \Lambda_A[f],[g]\rangle := B^s_A(P_{A}f,g),$$ 
where $f,g\in H^s(\mathbb R^n)$. This can be proved making use of the well-posedness of the direct problem and the properties of the bilinear form (see e.g. \cite{GSU20}, \cite{CMRU22}, \cite{CdHS22} for the standard proof). In particular, the self-adjointness follows from the symmetry of the bilinear form $B^s_A$, which is due to formula \eqref{symmetry}. Equation \eqref{equivalence-bilinear-forms} now ensures the symmetry of the bilinear form $B^s_Q$ in $\Phi H^s(\R^n)$ (which may be non-trivial at first sight).
\\
The DN map $\Lambda_A$ is related to the transformed potential $Q$ by the following integral identity:

\begin{proposition}[Alessandrini identity]\label{alex}
Let $\Omega\subset\mathbb R^n$ be a bounded open set and $s\in(0,1)$. Let $A_1, A_2$ be two anisotropic conductivity matrices satisfying assumptions \ref{ass-exterior}--\ref{ass-Phi} and $A_1\sim A_2$. For $j=1,2$, let $\Phi_j, Q_j$ be the sequence and transformed potential corresponding to $A_j$. Then the following integral identity holds for all $f_1, f_2 \in C^\infty_c(\Omega_e)$
$$\langle (\Lambda_{A_1}-\Lambda_{A_2})[f_1],[f_2]\rangle = \langle w_1\tridots (Q_2-Q_1), w_2 \rangle, $$
where $w_1 := \Phi_1P_{A_1}f_1$ and $w_2 := \Phi_2P_{A_2}f_2$. 
\end{proposition}

\begin{proof}
By definition of the DN maps and formula \eqref{equivalence-bilinear-forms} we have
\begin{align*}
    \langle (\Lambda_{A_1}-\Lambda_{A_2})[f_1],[f_2]\rangle & = \langle\Lambda_{A_1}[f_1],[f_2]\rangle - \langle \Lambda_{A_2}[f_1],[f_2]\rangle\\ & = \langle \Lambda_{A_1}[f_1],[f_2]\rangle - \langle \Lambda_{A_2}[f_2],[f_1]\rangle \\ & = B^s_{A_1}(P_{A_1}f_1, f_2) - B^{s}_{A_2}(P_{A_2}f_2, f_1)  \\ & = B^s_{Q_1}(\Phi_1 P_{A_1}f_1, \Phi_1 f_2) - B^s_{Q_2}(\Phi_2 P_{A_2}f_2, \Phi_2 f_1)
\end{align*}
Observe that $P_{A_2}f_2-f_2 \in \widetilde H^s(\Omega)$ by the definition of the Poisson operator. Thus $\Phi_2(P_{A_2}f_2-f_2) \in \Phi_2\widetilde H^s(\Omega)$, and by the assumption $A_1\sim A_2$ we can also deduce $\Phi_2(P_{A_2}f_2-f_2) \in \Phi_1\widetilde H^s(\Omega)$. This implies $B^s_{Q_1}(\Phi_1P_{A_1}f_1, \Phi_2(P_{A_2}f_2-f_2))=0$. Since moreover $A_1\sim A_2$ implies $\Phi_1 = \Phi_2$ in $\Omega_e$, we eventually deduce
$$ B^s_{Q_1}(\Phi_1P_{A_1}f_1, \Phi_1 f_2) = B^s_{Q_1}(\Phi_1P_{A_1}f_1, \Phi_2 f_2) = B^s_{Q_1}(\Phi_1P_{A_1}f_1, \Phi_2P_{A_2}f_2), $$
and similarly for $B^s_{Q_2}$. Therefore
\begin{align*}
    \langle (\Lambda_{A_1}-\Lambda_{A_2})[f_1],[f_2]\rangle & = B^s_{Q_1}(\Phi_1P_{A_1}f_1, \Phi_2P_{A_2}f_2) - B^s_{Q_2}(\Phi_2P_{A_2}f_2, \Phi_1P_{A_1}f_1) \\ & = \langle w_1 \tridots (Q_1-Q_2), w_2 \rangle,
\end{align*}
where we used the symmetry of $B^s_{Q_2}$ and set $w_1 := \Phi_1P_{A_1}f_1$ and $w_2 := \Phi_2P_{A_2}f_2$.
\end{proof}

\section{Proof of the main theorem}\label{sec-mainproof}
We now move forward to proving a Runge-type approximation result. The proof is based on the UCP for the fractional Laplacian and a standard technique (see \cite{GSU20,CMRU22}).

\begin{lemma}[Runge approximation property]\label{better-runge-2}
Let $\Omega, W\subset\mathbb R^n$ be bounded open sets such that $W\subset\Omega_e$. Assume that the exterior coefficients are isotropic, i.e. 
\begin{enumerate}[label=(A\arabic*)] \setcounter{enumi}{3}
    \item\label{ass-isotropic} there exist scalar functions $b_k$ such that $\beta_k=b_k Id$ for all $k\in\N$.
\end{enumerate}
The set $$\mathcal R := \left\{ P_Af-f :  f\in C^\infty_c(W) \right\}$$ is dense in $L^2(\Omega)$.
\end{lemma}

\begin{proof} Assume that $F\in L^2(\Omega)$ is such that $\langle F,v\rangle=0$ for all $v\in\mathcal R$, and let $u\in \widetilde H^s(\Omega)$ be the unique solution of \eqref{original2} with vanishing exterior value. Moreover, let $w\in \Phi\widetilde H^s(\Omega)$ correspond to $u$ by the fractional Liouville reduction. Then for all $f\in C^\infty_c(W)$ 
\begin{align*}
    0& = \langle F,  P_Af-f\rangle  = \langle \mathbf C^s_A u,  P_Af-f\rangle = B^s_A(u,  P_Af-f)  = -B^s_A(u, f)  = -B^s_Q(\Phi u,  \Phi f).
\end{align*}
Since the supports of $u$ and $f$ have empty intersection, we are left with $$ 0=\langle(-\Delta)^{s-1}Dw, \Phi f\rangle= \langle\Phi \tridots (-\Delta)^{s-1} Dw, f\rangle, $$
which by the arbitrariety of $f\in C^\infty_c(W)$ implies the vanishing of $\Phi \tridots (-\Delta)^{s-1} Dw$ in $\Omega_e$.
Recall that by assumption \ref{ass-exterior} in $\Omega_e$ we have $\Phi_{2k+1}(x) = 0$ and $\Phi_{2k}(x)=\beta_k(x)$ for all $k\in\N$. Thus in $\Omega_e$ it holds
\begin{align*}
    0 = \Phi \tridots (-\Delta)^{s-1} Dw = \sum_{k} \Phi_k : (-\Delta)^{s-1}D (\Phi_k u) = \sum_{k}\beta_k :(-\Delta)^{s-1}D (\beta_k u).
\end{align*}

Let $\tilde\beta_k$ be the constant exterior value of $\beta_k$, and define $$\gamma(x) := \sum_{k}\tilde\beta_k\odot\beta_k(x) =\sum_{k}\beta_k(y)\odot\beta_k(x) = \tilde a(x,y)$$ for all $x\in\R^n$ and $y\in\Omega_e$. We have that in $\Omega_e$ \begin{align*}
    0&=(-\Delta)^{s-1}\left(\sum_{k} \tilde\beta_k:D(\beta_ku)\right)\\ & = -\frac{1}{n+2s}(-\Delta)^{s-1}\left(\sum_{k}\tilde\beta_k:(\Delta Id + 2s\nabla^2)\odot (\beta_ku)\right)
    \\ & = -\frac{1}{n+2s}(-\Delta)^{s-1}(\Delta Id + 2s\nabla^2):\left(\sum_{k}\tilde\beta_k\odot \beta_ku\right) \\ & = -\frac{(-\Delta)^{s-1}(\Delta Id + 2s\nabla^2):\left(\gamma u\right)}{n+2s},
\end{align*}
and thus $(-\Delta)^{s-1}(\Delta Id + 2s\nabla^2):(\gamma u)=0$ in $\Omega_e$. Since $(-\Delta)$ and $(\Delta Id + 2s\nabla^2):$ are local operators, and also $u=0$ in $\Omega_e$, we deduce $$(-\Delta)^{s}[(\Delta Id + 2s\nabla^2):(\gamma u)]=0, \qquad (\Delta Id + 2s\nabla^2):(\gamma u)=0 $$ in $\Omega_e$. Therefore $(\Delta Id + 2s\nabla^2):(\gamma u)=0$ in all $\R^n$ by the UCP for the fractional Laplacian. 

Since $\gamma(x)$ is a symmetric matrix for all $x\in\R^n$, we can compute
\begin{align*}
    \nabla^2:(\gamma u) & = \sum_{i,j}\partial_i \partial_j(\gamma_{ij}u) \\ & = \sum_{i,j} \partial_i\partial_j\gamma_{ij}u + \partial_i\gamma_{ij}\partial_ju + \partial_j\gamma_{ji}\partial_iu + \gamma_{ij}\partial_i\partial_ju \\ & = \gamma:\nabla^2u + 2(\nabla\cdot\gamma)\cdot\nabla u + (\nabla^2:\gamma) u
\end{align*}
and
\begin{align*}
\Delta Id:(\gamma u) &= \Delta(u \mbox{ tr}\gamma) \\ & = \mbox{ tr}\gamma \Delta u + 2(\nabla \mbox{tr}\gamma) \cdot \nabla u + (\Delta\mbox{tr}\gamma) u \\ & =  (\mbox{tr}\gamma\, Id): \nabla^2 u + 2(\nabla \mbox{tr}\gamma) \cdot \nabla u + (\Delta\mbox{tr}\gamma) u,
\end{align*}
which gives 
\begin{align*}
0 & = (\mbox{tr}\gamma\, Id + 2s\gamma): \nabla^2 u + (2\nabla \mbox{tr}\gamma + 4s\nabla\cdot\gamma) \cdot \nabla u + (\Delta\mbox{tr}\gamma + 2s\nabla^2:\gamma) u \\ & =  \Gamma: \nabla^2 u + (2\nabla\cdot\Gamma) \cdot \nabla u + (\nabla^2:\Gamma) u \\ & = \nabla^2:(\Gamma u) := Lu 
\end{align*} 
in all $\R^n$, with the definition $\Gamma:= \mbox{tr}\gamma\, Id + 2s\gamma$. 

Only at this point we use the assumption that the exterior coefficients are isotropic, i.e. that there exist scalar functions $b_k$ such that $\beta_k=b_k Id$ for all $k\in\N$. This gives $$\gamma(x) = Id\sum_{k}\tilde b_k b_k(x) \qquad \mbox{ and } \qquad \Gamma(x):=(n+2s)\, Id \sum_{k}\tilde b_k b_k(x),$$
and therefore
$$ 0=Lu(x) = \nabla^2:(\Gamma(x) u(x)) = (n+2s)\Delta( u(x)\sum_{k}\tilde b_k b_k(x) ), $$
for all $x\in \R^n$. The function $u\sum_{k} \tilde b_k b_k$ is thus seen to be harmonic, and since it vanishes on $\Omega_e$ it must be equal to $0$ in all of $\R^n$. However, given that by assumption \ref{ass-positive} we have $\sum_{k} \tilde b_k b_k(x)>0$ for all $x\in \R^n$, the function $u$ must itself vanish everywhere.

We have proved that any $F\in L^2(\Omega)$ such that $\langle F,v\rangle=0$ for all $v\in\mathcal R$ vanishes identically. This gives the wanted result by the Hahn-Banach theorem.

\end{proof}

\begin{remark}
The above result can be generalized to a larger class of exterior coefficients. In particular, any exterior coefficients $\beta_k$ such that the weak maximum principle holds for the operator $L$ would suffice. Let us make for the sake of this discussion the slightly stronger regularity assumption that $\Phi\in C^{2+\eps}(\R^n)$. Then $\Gamma\in C^{2+\eps}(\R^n)$, and all the coefficients of the operator $L$ belong to $C^\eps(\R^n)$. Moreover, $L$ is uniformly elliptic, since for all $\xi,x\in\R^n$ and $y\in\Omega_e$ 
\begin{align*}
    \Gamma(x):(\xi\otimes\xi) &= (\mbox{tr}\gamma(x)\, Id + 2s\gamma(x) ):(\xi\otimes\xi) \\ & = \mbox{tr }\tilde a(x,y)|\xi|^2 + 2s\tilde a(x,y):(\xi\otimes\xi) \\ & \geq (n\lambda_{min}(x,y) + 2s\nu)  |\xi|^2 \\ & \geq (n+2s)\nu|\xi|^2.
\end{align*}
Here we used assumptions \ref{ass-exterior}, \ref{ass-positive} and observed that $\nu \in (0, \lambda_{min}(x,y)]$ for all $x,y\in \R^n$, where $\lambda_{min}(x,y)$ is the minimal eigenvalue of $\tilde a(x,y)$. 
Due to the H\"older continuity of the coefficients of $L$, $u\in C^{2+\eps}(\R^n)$ is actually a classical solution. If now the weak maximum principle holds for the operator $L$, we may consider equation $Lu=0$ in a large enough ball $B$ containing $\Omega$, and the result would follow from the fact that $u$ vanishes on $\partial B$.  
\end{remark}

We are now ready to give the proof of our main result.

\begin{proof}[Proof of Theorem \ref{main-theorem}]
Assume without loss of generality that $W_1$ and $W_2$ are disjoint, as this condition can always be achieved by restriction of the data. Let $f_j\in C^\infty_c(W_j)$ and $u_j := P_{A_j}f_j$ for $j=1,2$. Using the Alessandrini identity and the assumption on the DN maps, we can write
\begin{equation}\begin{split}
    0 &=\langle(\Lambda_{A_1}-\Lambda_{A_2})[f_1],[f_2]\rangle \\ &  = \langle \Phi_1u_1 \tridots (Q_1-Q_2), \Phi_2u_2 \rangle \\ & = \langle \Phi_1\tridots (Q_1-Q_2)\tridots \Phi_2, u_1u_2 \rangle.
\end{split}\end{equation} 
We now apply the Runge approximation property. For all $g_1,g_2 \in C^\infty_c(\Omega)$ we can find two sequences of exterior values $\{f_{j,i}\}_i\subset C^\infty_c(W_j)$, $j=1,2$, such that
$$  u_{j,i} :=  f_{j,i}+ g_j + r_{j,i} ,\qquad \mbox{with} \quad \|r_{j,i}\|_{L^2(\Omega)} \leq 1/i.$$
Substituting these solutions into the previous equation gives
\begin{align*}
    0 & = \langle \Phi_1\tridots (Q_1-Q_2)\tridots \Phi_2, u_{1,i}u_{2,i} \rangle \\ & = \langle \Phi_1\tridots (Q_1-Q_2)\tridots \Phi_2, (f_{1,i}+ g_1 + r_{1,i})(f_{2,i}+ g_2 + r_{2,i}) \rangle \\ & = \langle \Phi_1\tridots (Q_1-Q_2)\tridots \Phi_2, (g_1 + r_{1,i})(g_2 + r_{2,i}) \rangle
\end{align*}
where we used the support assumptions and the locality of the operators involved. Since $\Phi_1, \Phi_2 \in C^{2s+\eps}(\R^n)$ and $Q_1, Q_2 \in L^\infty(\R^n)$ (as showed in the proof of Lemma \ref{reduction-lemma}), we have  
\begin{align*}
    |\Phi_1\tridots (Q_1-Q_2)\tridots \Phi_2| & = \left| \sum_{k,i,j,k',i',j'} \Phi_{1,k,ij}  (Q_1-Q_2)_{k,ij,k',i'j'} \Phi_{2,k',i'j'} \right| \\ & \leq \left(\sum_{k,i,j,k',i',j'}|(Q_1-Q_2)_{k,ij,k',i'j'}|^2 \right)^{1/2}\left(\sum_{k,i,j,k',i',j'}|\Phi_{1,k,ij}|^2|\Phi_{2,k',i'j'}|^2 \right)^{1/2} \\ & \leq \|Q_1-Q_2\|_{L^\infty}\|\Phi_1\|_{C^{2s+\eps}}\|\Phi_2\|_{C^{2s+\eps}}<\infty.
\end{align*}
Thus $ \Phi_1\tridots (Q_1-Q_2)\tridots \Phi_2 \in L^\infty(\R^n)$, and the error terms vanish in the limit by the Cauchy-Schwartz inequality. Therefore, 
\begin{align*}
    0 & = \langle \Phi_1\tridots (Q_1-Q_2)\tridots \Phi_2, g_1g_2 \rangle,
\end{align*} and the arbitrariety of $g_1, g_2 \in C^\infty_c(\Omega)$ ensures $\Phi_1\tridots (Q_1-Q_2)\tridots \Phi_2=0$ in $\Omega$.
\\

 Since $A_1\sim A_2$, there exists a function $\rho\in C^\infty_c(\Omega)$ such that $\Phi_2 = (\rho+1)\Phi_1$. Thus
\begin{align*}
    Q_1-Q_2 & = \frac{\Phi_2\otimes (-\Delta)^{s-1}D\Phi_2}{|\Phi_2|^2}-\frac{\Phi_1\otimes (-\Delta)^{s-1}D\Phi_1}{|\Phi_1|^2} \\ & = \frac{\Phi_1\otimes (-\Delta)^{s-1}D((\rho+1)\Phi_1)}{(\rho+1)|\Phi_1|^2}-\frac{\Phi_1\otimes (-\Delta)^{s-1}D\Phi_1}{|\Phi_1|^2} \\ & = \Phi_1\otimes \frac{ (-\Delta)^{s-1}D(\rho\Phi_1)-\rho(-\Delta)^{s-1}D\Phi_1}{(\rho+1)|\Phi_1|^2} ,
\end{align*}
and so in $\Omega$ it holds that
\begin{align*}
    0 &= \Phi_1\tridots (Q_1-Q_2)\tridots \Phi_2 \\ & = \left\{  (-\Delta)^{s-1}D(\rho\Phi_1)-\rho(-\Delta)^{s-1}D\Phi_1\right\} \tridots \Phi_1 \\ & = ((-\Delta)^{s-1}D(\rho\Phi_1) + (\rho\Phi_1)\tridots Q_1)\tridots \Phi_1.
\end{align*}
This means that $\rho\Phi_1$ weakly solves the transformed problem \eqref{transformed-problem} in $\Phi_1\widetilde H^s(\Omega)$ with vanishing inhomogeneity and exterior value. Since this problem is well-posed, it must be $\rho\Phi_1\equiv 0$, which entails $\Phi_1\equiv\Phi_2$. By formula \eqref{decomposition}, we eventually obtain
$$A_{1,s}(x,y) = \sum_k \Phi_{1,k}(x) \odot \Phi_{1,k}(y) = \sum_k \Phi_{2,k}(x) \odot \Phi_{2,k}(y) = A_{2,s}(x,y).$$
\end{proof}

\section{The limit case $s\rightarrow 1$}
In this section we consider the limit case $s\rightarrow 1$ for the fractional anisotropic conductivity operator. The following Proposition shows that, as expected, under slightly stronger conditions one recovers the classical anisotropic conductivity operator:

\begin{proposition}
Let $K\subset\R^n$ compact, and assume that $A\in L^\infty(\R^{2n},\R^{n\times n})$ verifies assumptions \ref{ass-exterior}, \ref{ass-positive} and
\vspace{2mm}
\begin{enumerate}[label=(A\arabic*)*] \setcounter{enumi}{2}
    \item\label{ass-star} there exists $\eps > 0$ such that $\Phi\in C^{2+\eps}(\R^n)$.
\end{enumerate}
\vspace{2mm}

\noindent Then $A':\R^n \rightarrow \R^{n\times n}$ defined for all $x\in\R^n$ by 
$$ A'(x):= \frac{Id \emph{ tr} A_s(x,x) + 2 A_s(x,x)}{n+2} $$
belongs to $C^{2+\eps}(\R^n, \R^{n\times n})$, is constant in $\Omega_e$, symmetric, and uniformly positive definite with the same constant $\nu$ as $A_s$. Moreover, the following limit holds for all $u\in H^2_K$: \begin{align*}
    \lim_{s\rightarrow 1}\mathbf C^s_Au =-\nabla\cdot(A'\cdot \nabla u).
\end{align*}

\end{proposition}

\begin{proof}
Since assumption \ref{ass-Phi} is substituted by \ref{ass-star}, $\Phi$ is a multiplier of $H^2(\R^n)$ into itself. The proof of this fact follows the one of Lemma \ref{holder-sobo-lemma}. Therefore, $\Phi u\in H^2_K$ and $D(\Phi u) \in L^2(\R^n)$ with support in $K$. Moreover, the facts that $\Phi_{2k+1}$ has compact support in $\Omega$ and that $\beta_k$ is constant in $\Omega_e$ for all $k\in\N$ imply that $D\Phi\in L^2(\R^n)$ with compact support.

If $v\in L^2(\R^n)$ with compact support, then $\lim_{s\rightarrow 1^-} (-\Delta)^{s-1}v=v$ almost everywhere. This can be proved directly by Fourier methods first for $v\in \mathcal S(\R^n)$, the set of Schwartz functions, then for general $v$ by density. 
In particular, $(-\Delta)^{s-1}D(\Phi u)\rightarrow D(\Phi u)$ and $(-\Delta)^{s-1}D\Phi \rightarrow D\Phi$ as $s\rightarrow 1^-$ almost everywhere in $\R^n$. Therefore, by the reduction Lemma \ref{reduction-lemma} we compute
\begin{align*}
    \lim_{s\rightarrow 1}\mathbf C^s_Au & = \lim_{s\rightarrow 1} \Phi \tridots \left\{(-\Delta)^{s-1}D(\Phi u) - u (-\Delta)^{s-1}D\Phi\right\} \\ & = -\frac{\Phi}{n+2} \tridots \left\{ \left(\Delta Id + 2\nabla^2\right)\odot(\Phi u) - u \left(\Delta Id + 2\nabla^2\right)\odot\Phi\right\} \\ & = -\sum_k \frac{\Phi_k}{n+2} : \left\{ \left(\Delta Id + 2\nabla^2\right)\odot(\Phi_k u) - u \left(\Delta Id + 2\nabla^2\right)\odot\Phi_k\right\} \\ & =-\sum_k \frac{\Phi_k}{n+2} : \left\{ Id \odot \left(\Delta (\Phi_k u) - u \Delta \Phi_k\right)\right\} -\sum_k \frac{2\Phi_k}{n+2} : \left\{ \nabla^2\odot(\Phi_k u) - u\nabla^2\odot\Phi_k\right\} \\ & =: S_1+S_2 .
\end{align*}
For the first term $S_1$ we write
\begin{align*}
    -(n+2)S_1 & = \sum_k \Phi_k : \left\{ Id \odot \left(\Phi_k \Delta u + 2\nabla u \cdot \nabla \Phi_k \right)\right\} \\ & = Id :\sum_k \left\{ \left(\Phi_k \Delta u + 2\nabla u \cdot \nabla \Phi_k \right)\odot \Phi_k\right\} \\ & = Id :\sum_k \left\{ \left(\Phi_k\odot\Phi_k\right) \Delta u + \nabla u \cdot \nabla \left(\Phi_k \odot \Phi_k\right)\right\} \\ & = Id :\left\{ A_s(x,x) \Delta u + \nabla u \cdot \nabla A_s(x,x)\right\} \\ & =  \mbox{tr}A_s(x,x) \Delta u +\nabla u \cdot \nabla \mbox{tr}A_s(x,x) \\ & = \nabla\cdot(\mbox{tr}A_s(x,x) \nabla u).
\end{align*}
Similarly, for the second term $S_2$ we see that
\begin{align*}
    -(n+2)S_2 & = 2\sum_k \sum_{i,j} \Phi_{k,ij} \left\{ \nabla^2\odot(\Phi_k u) - u\nabla^2\odot\Phi_k\right\}_{ij} \\ & = 2\sum_k \sum_{i,j} \Phi_{k,ij} \left\{ \partial_i\partial_j(\Phi_{k,ij} u) - u\partial_i\partial_j\Phi_{k,ij}\right\}  \\ & = 2\sum_k \sum_{i,j} \Phi_{k,ij}^2\partial_i\partial_j u + \Phi_{k,ij}\partial_i\Phi_{k,ij}\partial_ju + \Phi_{k,ij}\partial_j\Phi_{k,ij}\partial_iu \\ & = \sum_{i,j} 2(\sum_k\Phi_{k,ij}^2)\partial_i\partial_j u + \partial_i(\sum_k\Phi_{k,ij}^2)\partial_ju + \partial_j(\sum_k\Phi_{k,ij}^2)\partial_iu.
\end{align*}
However, $$ \sum_k\Phi_{k,ij}^2 = \{\sum_k\Phi_{k}\odot \Phi_k\}_{ij} = A_s(x,x)_{ij} = A_s(x,x)_{ji},$$
and thus 
\begin{align*}
    -(n+2)S_2 & = 2\sum_{i,j} A_s(x,x)_{ij}\partial_i\partial_j u + \partial_ju\,\partial_iA_s(x,x)_{ij} \\ & = 2\sum_{i,j} \partial_i (A_s(x,x)_{ij}\partial_j u) \\ & = 2\nabla\cdot(A_s(x,x)\cdot\nabla u).
\end{align*}
This gives the expected formula for $A'$. Since $A_s\in C^{2+\eps}(\R^{2n})$, the function $x\mapsto A_s(x,x)$ is itself in $C^{2+\eps}(\R^n)$,
and we have $A'\in C^{2+\eps}(\R^n)$. The facts that $A'$ is symmetric and that it is constant in $\Omega_e$ follow from the related properties of $A_s$. The new matrix $A'$ is also uniformly positive definite, since for all $x, \xi\in \R^n$ it holds 
\begin{align*}
    A'(x):(\xi\otimes\xi) & = \frac{Id \mbox{ tr} A_s(x,x) + 2 A_s(x,x)}{n+2}: (\xi\otimes\xi) \\ & = \frac{ \mbox{ tr} A_s(x,x) |\xi|^2 + 2 A_s(x,x):(\xi\otimes\xi)}{n+2} \\ & \geq \frac{ n \nu |\xi|^2 + 2 \nu |\xi|^2}{n+2} = \nu |\xi|^2. 
\end{align*}
The inequality in the above computation is due to \ref{ass-positive} and the consideration that $\nu \in (0, \lambda_{min}(x,y)]$ for all $x,y\in \R^n$, where $\lambda_{min}(x,y)$ is the minimal eigenvalue of $A_s(x,y)$. 

\end{proof}

\begin{remark}
When the matrix $A_s$ is isotropic, i.e. of the form $\sigma(x,y)$Id for a scalar function $\sigma$, then the limit matrix $A'$ is
$$ A'(x):= \frac{Id \emph{ tr} A_s(x,x) + 2 A_s(x,x)}{n+2} = \frac{Id (n\sigma(x,x)) + 2 \sigma(x,x)Id}{n+2} = \sigma(x,x) Id = A_s(x,x), $$
and thus it is itself isotropic. In this case the above result can be inverted, since every isotropic, positive definite matrix $A'(x) = \gamma(x)$Id can be realized as the limit of a separable, isotropic matrix $A_s(x,y) = \gamma^{1/2}(x)\gamma^{1/2}(y)$Id. There exist however many different, non-separable matrices $A_s$ giving rise to the same limit matrix $A'$.

\end{remark}

\bibliography{refs-mathscinet} 

\bibliographystyle{alpha}

\end{document}